\documentclass[11pt,letterpaper]{amsart}

\usepackage{tikz, tikz-3dplot}
\usepackage{amsmath,amsthm,amssymb,graphics}
\usepackage[utf8]{inputenc}
\usepackage[english]{babel}
\usepackage[colorlinks=true, allcolors=blue]{hyperref}
\usepackage{slashbox}
\usepackage[title]{appendix}
\usepackage{multirow}
\usepackage{setspace}
\usepackage{comment}

\usepackage{marginnote}
\usepackage[textcolor=blue,
    linecolor=blue!10!white,
    bordercolor=blue!10!white,
    backgroundcolor=white, ]{todonotes}
\numberwithin{equation}{section}

\theoremstyle{plain}

\newtheorem{te}{Theorem}[section]

\newtheorem{prop}[te]{Proposition}

\newtheorem{lem}[te]{Lemma}

\newtheorem*{ack*}{Acknowledgment}
\theoremstyle{remark}

\newtheorem{rmk}{Remark}

\newcommand{\dsum}{\displaystyle\sum}
\newcommand{\dint}{\displaystyle\int}

\newcommand{\nocontentsline}[3]{}

\def\x{{\boldsymbol x}}
\def\y{{\bf y}}
\def\z{{\bf z}}

\def\b{{\boldsymbol b}}
\def\t{{\bf t}}

\def\0{{\bf 0}}
\def\a{{\boldsymbol a}}
\def\b{{\boldsymbol b}}
\def\d{{\boldsymbol d}}
\def\c{{\boldsymbol c}}

\def\h{{\boldsymbol h}}
\def\y{{\boldsymbol y}}
\def\z{{\boldsymbol z}}

\def\r{{\boldsymbol r}}

\def\R{{\mathbb R}}
\def\Q{{\mathbb Q}}

\def\N{{\mathbb N}}
\def\C{{\mathbb C}}

\def\Z{{\mathbb Z}}
\def\P{{\mathbb P}}\def\A{{\mathbb A}}

\newcommand{\bA}{\mathbb{A}}
\DeclareMathOperator{\codim}{codim}

\begin{document}

\author{Heejong Lee}
\address[H.~Lee]{Department of Mathematics, Purdue University, 150 N. University Street, West Lafayette, IN 47907-2067, USA}
    \email{\href{mailto:lee4878@purdue.edu}{lee4878@purdue.edu}}

\author{Seungsu Lee}
\address[S.~Lee]{Department of Mathematics, University of Michigan, 4839 East Hall,
530 Church Street,
Ann Arbor, MI 48109-1043, USA}
    \email{\href{mailto:starlee@umich.edu}{starlee@umich.edu}}
    
\author{Kiseok Yeon}

\address[K.~Yeon]{Department of Mathematics, Purdue University, 150 N. University Street, West Lafayette, IN 47907-2067, USA}
\email{\href{mailto:kyeon@purdue.edu}{kyeon@purdue.edu}}

\title[Heejong Lee, Seungsu Lee, Kiseok Yeon]{The local solubility for homogeneous polynomials with random coefficients over thin sets}
\maketitle

\begin{abstract}

   Let $d$ and $n$ be natural numbers greater or equal to $2$. Let $\langle \boldsymbol{a}, \nu_{d,n}(\boldsymbol{x})\rangle\in \mathbb{Z}[\boldsymbol{x}]$ be a homogeneous polynomial in $n$ variables of degree $d$ with integer coefficients $\boldsymbol{a}$, where $\langle\cdot,\cdot\rangle$ denotes the inner product, and $\nu_{d,n}: \mathbb{R}^n\rightarrow \mathbb{R}^N$ denotes the Veronese embedding with $N=\binom{n+d-1}{d}$. Consider a variety $V_{\boldsymbol{a}}$ in $\mathbb{P}^{n-1}$, defined by $\langle \boldsymbol{a}, \nu_{d,n}(\boldsymbol{x})\rangle=0.$ In this paper, we examine a set of integer vectors $\boldsymbol{a}\in \Z^N$, defined by
    $$\mathfrak{A}(A;P)=\{ \boldsymbol{a}\in \Z^N:\ P(\boldsymbol{a})=0,\ \|\boldsymbol{a}\|_{\infty}\leq A\},$$
    where $P\in \mathbb{Z}[\x]$ is a non-singular form in $N$ variables of degree $k$ with $2 \le k\leq C({n,d})$ for some constant $C({n,d})$ depending at most on $n$ and $d$. 
    Suppose that $P(\boldsymbol{a})=0$ has a nontrivial integer solution. We confirm that the proportion of integer vectors $\boldsymbol{a}\in \Z^N$ in $\mathfrak{A}(A)$, whose associated equation $\langle \boldsymbol{a}, \nu_{d,n}(\boldsymbol{x})\rangle=0$ is everywhere locally soluble, converges to a constant $c_P$ as $A\rightarrow \infty.$ Moreover, for each place $v$ of $\Q$, if there exists a non-zero $\boldsymbol{b}_v\in \Q_v^N$ such that $P(\boldsymbol{b}_v)=0$ and the variety $V_{\boldsymbol{b}_v}$ in $\mathbb{P}^{n-1}$ admits a smooth $\mathbb{Q}_v$-point, the constant $c_P$ is positive.
    
   % We confirm that when $d\geq 4$, $n$ is sufficiently large in terms of $d$, and $k\leq d,$ the proportion of elements in $\mathbb{V}^{P}_{d,n}(A)$, which do not satisfy the Hasse principle, converges to $0$ as $A\rightarrow \infty$. We make explicit a lower bound on $n$ guaranteeing this conclusion, and in particular, show that when $d\geq 14$ it suffices to take $n\geq 32d+17$. In order to establish it, we mainly use the Hardy-Littlewood circle method. 
    
 %  In comparison with the recent work due to Browing, Le Boudec, and Sawin (2022), we are capable of obtaining information about the distribution of the exceptional set of $V_\a$ with $\|\a\|_{\infty}\leq A$ beyond bounds on their cardinality. In particular, the exceptional set of $V_\a$ cannot take up positive density within a set of varieties $\{V_{\a}\in \P^{n-1}|\ P(\a)=0\}$, provided that $n>64d$.
 
   %Furthermore, as an independent interest, when $N$ is sufficiently large in terms of $k$ and $W\leq A^{1/2}$, we provide essentially optimal bound for the following mean value
  % $$\{\a_1,\a_2\in [-A,A]^N\cap \Z^N|\ P(\a_1)=P(\a_2)=0,\ \a_1\equiv\a_2\ (\text{mod}\ W)\}.$$
\end{abstract}

%\addtocontents{toc}{\protect\setcounter{tocdepth}{0}}
\setcounter{tocdepth}{1}
\tableofcontents

\section{Introduction}\label{sec1}

In this article, we study the $p$-adic and real solubility for projective varieties defined by forms with integer coefficients. Denote by $f(\x)\in \Z[x_1,x_2,\ldots,x_n]$ a homogeneous polynomial of degree $d$ and let us write $V\subset \P^{n-1}$ for a projective variety defined by $f(\x)=0$. %A conjecture of Artin [$\ref{ref70}$] asserts that the variety $V$ admits a $\Q_p$-point for every $p$ if $n\geq d^2+1$. %Then, one may expect that whenever $n\geq d^2+1$ the variety $V$ admits a $\Q_p$-point for every $p$, which can be achieved by verifying the truth of the Artin's Conjecture [$\ref{ref70}$] asserting that for a given form $F(\x)\in \Q_p[x_1,\ldots,x_n]$, the equation $F(x_1,\ldots,x_n)=0$ has a solution in $\Q_p^n\setminus\{\0\}$ whenever $n\geq d^2+1.$ 
%However, this conjecture is known to be false in general, and the first counter example has been founded by Terjanian [$\ref{ref1012}$]. % the current knowledge seems very far from this conjecture. %\textcolor{purple}{HL: it seems that there is a counter-example to Artin's conjecture. See Ax-Kochen theorem in wikipedia.}
%Alternatively, 
One may ask if there exists $n_0:=n_0(d)\in \N$ such that whenever $n> n_0(d),$ the variety $V$ admits a $\Q_p$-point. The current knowledge is accessible to this type of question.
In particular, it is known by Wooley [$\ref{ref102}$] that it suffices to take $n_0= d^{2^d}$ (see also  [$\ref{ref80}$], [$\ref{ref311}$], [$\ref{ref300}$]). As a different approach concerning the $p$-adic solubility for the variety $V$, the Ax--Kochen theorem [$\ref{ref701}$] shows that  a homogeneous polynomial $f(\x)\in \Z[x_1,\ldots,x_n]$ of degree $d$ with $n \geq d^2+1$  has a solution in $\Q_p^n \setminus \{\0\}$ for sufficiently large prime $p$ with respect to $d$ (see also $[\ref{ref801}]$). It is also known that if  $f(\x)\in \Z[x_1,x_2,\ldots,x_n]$ is an absolutely irreducible form of degree $d$ over $\mathbb{F}_p$ with a prime $p$ sufficiently large in terms of degree $d$, it follows by applying the Lang--Weil estimate (see [$\ref{ref100.2}$] and $[\ref{ref100.0}, \text{Theorem}\ 3]$) and the Hensel's lemma that the equation $f(\x)=0$ has a solution in $\Q_p^n\setminus\{\0\}$. As for the real solubility for the variety $V$, one sees that if the degree $d$ is odd, the variety $V$ always admits a real point. For the degree $d$ even, to the best of the authors' knowledge, it is not known how to determine, in general, whether the variety $V$ has a real point or not.

One infers from the previous paragraph that if the number of variables $n$ is not large enough, the main difficulties for verifying the $p$-adic solubility for the variety $V$ occur from the case of small prime $p.$ Nevertheless, thanks to the density lemma introduced in [$\ref{ref30}$, Lemma 20], we are capable of obtaining some information about the $p$-adic solubility for the variety $V$ even with small primes $p.$ In order to describe this information, we temporarily pause and introduce some notation. Let $n$ and $d$ be natural numbers with $d\geq 2. $ Let $N:=N_{d,n}=\binom{n+d-1}{d}.$ Let $\nu_{d,n}: \R^n\rightarrow \R^N$ denote the Veronese embedding, defined by listing all the monomials of degree $d$ in $n$ variables with the lexicographical ordering. %\memo{HL: simplified sentences below} 
Let $\langle\cdot,\cdot\rangle$ be the inner product on $\A^N$. We denote by $f_{\a}(\x):=\langle\a,\nu_{d,n}(\x)\rangle$ the homogeneous polynomial in $n$ variables of degree $d$ with coefficients $\a \in \Z^N$. Define $V\subset \P^{n-1}\times \A^N$ to be the subvariety given by $f_{\a}(\x)=0.$ Let $\pi: V\rightarrow \A^N$ be the projection onto the second factor. We denote by $V_{\a}:= \pi^{-1}(\a)$. For $A \in \R_{>0}$, we define
 \begin{equation}\label{set}
     \mathfrak{A}(A):=\{\a\in\Z^N:\ \a\in [-A,A]^N\}.
 \end{equation}
  %For any $V_{\a}$, write $V_{\a}\biggl(\R\times\displaystyle\prod_{p\ \text{prime}}\Z_p\biggr)$ for the set of points in $\R^n\times\displaystyle\prod_{p\ \text{prime}}\Z_p^n$ that the variety $V_{\a}$ admits. Define a quantity
  The following quantity
   \begin{equation}\label{density}
 \varrho_{d,n}^{\text{loc}}(A):=\frac{\# \biggl\{\a\in  \mathfrak{A}(A) :\ V_{\a}\biggl(\R\times\displaystyle\prod_{p\ \text{prime}}\Z_p\biggr)\neq \emptyset \biggr\}}{\# \mathfrak{A}(A)}.
 \end{equation}
 %We notice here that the quantity $\varrho_{d,n}^{\text{loc}}(A)$ 
 is the proportion of $\a \in \mathfrak{A}(A)$ such that $V_{\a}$ is \textit{locally soluble}, i.e.~that admit a real point and a $p$-adic point for all primes $p$. %\memo{HL: Revise ``Thus ... with small primes $p$.''} 
 Thus, the behavior of $\varrho_{d,n}^{\text{loc}}(A)$ provides the information about $p$-adic solubility of the varieties $V_{\a}$ for $\a\in \mathfrak{A}(A),$ even with small primes $p$. We also define $c_{\infty}$ (resp.~$c_p$) to be the measure of $\a$ in $[-1,1]^N$ (resp.~in $\Z_p^N$), whose associated variety $V_{\a}$ admits a real point (resp.~a $p$-adic point).    By using the density lemmas [$\ref{ref30}$, Lemmas 20 and 21], Poonen and Voloch [$\ref{ref14}$, Theorem 3.6] proved that whenever $n,d\geq 2$, one has
\begin{equation}\label{Poonen}
   \lim_{A\rightarrow \infty}\varrho_{d,n}^{\text{loc}}(A)=c,
\end{equation}
where $c$ is the product of $c_{\infty}$ and $c_p$ for all primes $p$.

In this paper, we investigate the proportion of $\a$ with locally soluble $V_{\a}$ in the type I thin set in the sense of Serre given by the vanishing locus of $P(\mathbf{t})\in \Z[t_1,\dots, t_N]$. Here, we take $P$ to be a non-singular form in $N$ variables of degree $k\geq 2$. We define 
\begin{align*}
    \mathfrak{A}(A;P) := \{\a \in \mathfrak{A}(A): P(\a)=0 \}.
\end{align*}
Analogous to the quantity $\varrho_{d,n}^{\text{loc}}(A)$, we have the proportion of $\a$ in the thin set with locally soluble $V_{\a}$ defined as
 \begin{equation}\label{def1.1}
 \varrho_{d,n}^{P,\text{loc}}(A):=\frac{\# \biggl\{\a \in \mathfrak{A}(A;P) :  V_{\a}\biggl(\R\times\displaystyle\prod_{p\ \text{prime}}\Z_p\biggr)\neq \emptyset \biggr\}}{\#\mathfrak{A}(A;P)}.
 \end{equation}
%as $A\rightarrow \infty$.

%\bigskip

 %In order to describe our main theorems, we define
%\begin{equation}\label{Tinfinity}
 %  T_{\infty}:=\left\{\a\in [-1,1]^N\cap \R^N\middle|\ \exists\ \x\in \R^n\setminus\{\boldsymbol{0}\}\ \text{such that}\ f_{\a}(\x)=0\right\} 
%\end{equation}
%and 
%\begin{equation}\label{Tp}
%    T_p:=\left\{\a\in \mathbb{Z}_p^N\middle|\  \exists\ \x\in \Z_p^n\setminus\{\boldsymbol{0}\}\ \text{such that}\ f_{\a}(\x)=0\right\},
%\end{equation}
%for every prime $p.$ 

Let $\mu_p$ denote the Haar measure on $\Z_p^N$ normalized to have total mass $1$. For given measurable sets $S_p\subseteq \Z_p^N$ and $S_{\infty}\subseteq \R^N$ with the Haar measure $\mu_p$ and the Lebesgue measure, we define
$$\sigma_{p}(S_p):=\lim_{r\rightarrow \infty}p^{-r(N-1)}\#\left\{\a\ (\text{mod}\ p^r)\middle|\ \a\in S_p\ \text{and}\ P(\a)\equiv0\ (\text{mod}\ p^r)\right\}$$
and
$$\sigma_{\infty}(S_{\infty})=\lim_{\eta\rightarrow 0+}(2\eta)^{-1}V_{\infty}(\eta),$$
where $V_{\infty}(\eta)$ is the volume of the subset of $\y\in S_{\infty}$ satisfying $|P(\y)|<\eta$. Furthermore, for $d_1, d_2 \in \N$, we define $C_{n,d}(d_1, d_2)$ to be the rational number such that \[C_{n,d}(d_1, d_2) =\frac{d!(n+d_1 - 1)!}{d_1!(n+d-1)!}+\frac{d!(n+d_2 - 1)!}{d_2!(n+d-1)!}.\]
Note that $C_{n,d}(d_1, d_2)$ belongs to $(0,1)$ and maximizes at $(1,d-1)$ and $(d-1,1)$; see Lemma \ref{lem.Cndmaximum}. Our first main theorem shows that the proportion $\varrho_{d,n}^{P,\text{loc}}(A)$ converges to the product of ``local proportions'' as $A \rightarrow \infty$.

%We define a set
%$$\mathcal{A}_{d,n}^P(A)=\{\a\in [-A,A]^N\cap \Z^N|\ P(\a)=0\}.$$ We introduce a set $\mathcal{L}^N$ of $\a\in \Z^N$ having a property that $\langle\a,\nu_{d,n}(\x)\rangle=0$ has a solution in $\R$ and $\Q_p$ for all primes $p.$ We define a set
%$$\mathcal{A}_{d,n}^{P,\textrm{loc}}(A)=\left\{\a\in [-A,A]^N\middle|\ \a\in \mathcal{L}^N\ \text{and}\  P(\a)=0\right\}$$
%In this paper, we examine the behavior of the quantity
%$$\frac{\#\mathcal{A}_{d,n}^{P,\textrm{loc}}(A)}{\#\mathcal{A}_{d,n}^P(A)},$$
%as $A\rightarrow \infty.$ 

\begin{te}\label{thm1.2}
Suppose that $P(\mathbf{t})\in \Z[t_1,\dots, t_N]$ is a non-singular form in $N$ variables of degree $k$ with $2 \le k< (1-C_{n,d}(1,d-1))N-1$ and $(k-1)2^k < N$. % where $C_{n,d}(d_1, d_2)$ is as above.  
Suppose that  $P(\mathbf{t})=0$ has a nontrivial integer solution. For $n,d\geq 2$ such that
\begin{align*}
    (n,d)\neq (2,2),(2,3),(3,2),(3,3)
\end{align*}
one has
$$\lim_{A\rightarrow\infty}\varrho_{d,n}^{P,\text{loc}}(A)=c_P,$$
where
$$c_P:=\frac{\sigma_{\infty}(\pi(V(\R)) \cap [-1,1]^N)\cdot\prod_p \sigma_p(\pi(V(\Q_p)\cap \Z_p^N)}{\sigma_{\infty}([-1,1]^N)\cdot\prod_p \sigma_p(\Z_p^N)}.$$
\end{te}
\begin{rmk}
Since $P(\t)$ is a non-singular form in $N$ variables of degree $k$ with $(k-1)2^k<N$ and $P(\t)=0$ has a nontrivial integer solution, the classical argument (see $[\ref{ref8}]$ and $[\ref{ref26}, \text{the proof of Theorem 1.3}]$) reveals that the quantity
    $ \sigma_{\infty}([-1,1]^N)\cdot\prod_{p\ \text{prime}} \sigma_p(\Z_p)$ is convergent and is bounded above and below by non-zero constants, respectively, depending on the polynomial $P(\t)$. Then, on observing that 
    $$0\leq \sigma_{\infty}(\pi(V(\R)) \cap [-1,1]^N)\leq \sigma_{\infty}([-1,1]^N)\ \text{and}\ 0\leq \sigma_p(\pi(V(\Q_p)\cap \Z_p^N)\leq \sigma_p(\Z_p),$$ 
    we infer that
    the infinite product $\sigma_{\infty}(\pi(V(\R)) \cap [-1,1]^N)\cdot\prod_p \sigma_p(\pi(V(\Q_p)\cap \Z_p^N)$ in the numerator of $c_P$ converges to a non-negative constant.
\end{rmk}
In [$\ref{ref3121}$, Corollary 1.6], Browning and Heath-Brown obtained the same constant for the case that $P$ is a quadratic form of rank at least $5.$ The modicum computation reveals that when $P$ is a non-singular quadratic form $(k=2)$, the conclusion of Theorem 1 holds for $N\geq 5$. Hence, we notice that the conclusion of Theorem 1.1 coincides with $[\ref{ref3121}$, Corollary 1.6] for non-singular quadratic forms. Furthermore, we emphasize that one could deal with $P$ a quadratic form of rank at least $5$ and obtain the same result in $[\ref{ref3121}$, Corollary 1.6], by using the argument described here. Additionally, we note that the argument in this paper seems plausible to be generalized for obtaining analogous results for thin sets defined by a system of non-singular forms.

Our second main theorem shows that under an additional condition on the polynomial $P$ defining the thin set, the product of local proportions is strictly positive.

\begin{te}\label{thm1.21.2}
    In addition to the setup of Theorem \ref{thm1.2}, suppose that for each place $v$ of $\Q$, there exists $\b_v\in \Q_v^N\setminus\{\0\}$ %(resp.~$\b_\infty \in \R^N\setminus \{0\}$) 
    such that the variety $V_{\b_v}$ in $\P^{n-1}_{\Q_v}$ %(resp.~$V_{\b_{\infty}}$ in $\P^{n-1}_\R$) 
    admits a smooth $\Q_v$-point. %(resp.~$\R$-point) for each prime $p$.  
    Then, the constant $c_P$ is positive.
    %Suppose that $P(\mathbf{t})\in \Z[t_1,\dots, t_N]$ is a non-singular form in $N$ variables of degree $k$ with $2\leq k\leq ?.$ Supoose that there exists $\b\in \Z^N$ with $P(\b)=0$ such that the variety $V_{\b}$ in $\P^{n-1}$ admits a smooth $\Q$-rational point. Then, whenever $n\geq 2$ and $d\geq 2$, one has
%$$\lim_{A\rightarrow\infty}\varrho_{d,n}^{P,\text{loc}}(A)>0.$$
\end{te}

\section*{Notation}

For a given vector $\boldsymbol{v}\in \R^N$, we write the $i$-th coordinate of $\boldsymbol{v}$ by $(\boldsymbol{v})_i$ or $v_i$. We use $\langle\cdot,\cdot\rangle$ for the inner product. We write $0\leq \x\leq X$ or $\x\in [0,X]^s$ to abbreviate the condition $0\leq x_1,\ldots,x_s\leq X$. For a prime $p$ and vectors $\boldsymbol{v}\in \R^n$, we use $p^h\|\boldsymbol{v}$ when one has $p^h|v_i$ for all $1\leq i\leq n$ but $p^{h+1}\nmid v_i$ for some $1\leq i\leq n.$ Throughout this paper, we use $\gg$ and $\ll$ to denote Vinogradov's well-known notation, and write $e(z)$ for $e^{2\pi iz}$. We use $A\asymp B$ when both $A\gg B$ and $A\ll B$ hold. We adopt the convention that when $\epsilon$ appears in a statement, then the statement holds for each $\epsilon>0$, with implicit constants depending on $\epsilon.$

\section*{Acknowledgement}
The authors acknowledge support from NSF grant DMS-2001549 under the supervision of Trevor Wooley. The third author would like to thank James Cumberbatch and Dong-ha Kim for helpful discussion. Especially, the third author would like to thank Trevor Wooley for his constant encouragement to complete this work. Furthermore, the authors would like to thank Daniel Loughran and Nicholas Rome for letting us know many related works and providing very helpful comments on overall arguments in this paper.
\addtocontents{toc}{\protect\setcounter{tocdepth}{2}}

%\bigskip

\section{Preliminaries}\label{sec2}

%\memo{HL: do we fix $P(\mathbf{t})$ throughout the paper? Then it is better to move this to the notation section.}
Throughout this section, we fix a non-singular form $P(\mathbf{t})\in \Z[t_1,\dots, t_N]$ in $N$ variables of degree $k\ge 2$. We also assume that $N>(k-1)2^k$. Let $\mathfrak{B}$ be a box in $[-1,1]^N\cap\R^N$.  For given $A,B>0$ and $\r \in\Z^N$ with $0\leq \r\leq B-1$, we define
$$\mathcal{N}(A,\mathfrak{B},B,\r,P)=\#\left\{\a\in A\mathfrak{B}\cap \Z^N\middle|\ P(\a)=0,\ \a\equiv \r\ (\text{mod}\ B)\right\}.$$ The first lemma in this section provides the asymptotic formula for $$\mathcal{N}(A,\mathfrak{B},B,\r,P) \ \text{as}\ A\rightarrow \infty.$$

In advance of the statement of Lemma $\ref{lem1.31.31.3}$, we define $v_p(L)$ with $L\in \Z$ and $p$ prime as the integer $s$ such that $p^{s}\| L.$
\begin{lem}\label{lem1.31.31.3}
    Suppose that $B$ is a natural number.  Then, for a given $\r\in \Z^N$ with $0\leq \r\leq B-1$ and for sufficiently large $A>0$, there exists $\delta>0$ such that 
    \begin{equation*}
    \begin{aligned}
      \mathcal{N}(A,\mathfrak{B},B,\r,P)=\prod_{p\nmid B}\sigma_{p}\cdot\prod_{p\mid B }\sigma^{B,\r}_{p}\cdot\sigma_{\infty}\cdot A^{N-k}+O(A^{N-k-\delta}), 
    \end{aligned}
    \end{equation*}
where %\memo{we used $\sigma_p$ for certain measure in the intro. Are they related? If yes, try to make it coherent. If not, we should replace $\sigma_p$ by something else here.}
\begin{equation*}
    \begin{aligned}
        \sigma_p&:=\lim_{l\rightarrow \infty}p^{-l(N-1)}\#\left\{1\leq \a\leq p^l:\ P(\a)\equiv0\ (\text{mod}\ p^l)\right\},\\
       \sigma^{B,\r}_{p}&:=\lim_{l\rightarrow \infty}p^{-l(N-1)}\#\left\{1\leq \a\leq p^l:\ P(\a)\equiv0\ (\text{mod}\ p^l),\ \a\equiv\r\ (\text{mod}\ p^{v_p(B)})\right\}
    \end{aligned}
\end{equation*}
and
$$\sigma_{\infty}:=\sigma_{\infty}(\mathfrak{B})=\lim_{\eta\rightarrow 0+}(2\eta)^{-1}V_{\infty}(\eta)$$
in which $V_{\infty}(\eta)$ is the volume of the subset of $\y\subseteq \mathfrak{B}$ satisfying $|P(\y)|<\eta.$  In particular, the implicit constant in $O(A^{N-k-\delta})$ depends on $B$ and $P(\mathbf{t}).$
\end{lem}

We record this lemma without proof because it is readily obtained by the previous results as follows. By repeating the argument of Birch in [$\ref{ref8}$] replaced with the variables $\x$ imposed on the congruence condition $\x\equiv\r\ (\text{mod}\ B)$, we obtain an asymptotic formula for the number of integer solutions $\x\in [-A,A]^N$ of $P(\x)=0$ with the congruence condition $\x\equiv\r\ (\text{mod}\ B)$ as $A\rightarrow \infty$.  The main term of this asymptotic formula includes the product of $p$-adic densities and the singular integral. By applying the strategy proposed by Schmidt [$\ref{ref9}$, $\ref{ref10}$] (see also a refined version [$\ref{ref11}$, section 9]), the singular integral can be replaced by the real density $\sigma_{\infty}$. Furthermore, we readily deduce by the coprimality between $p$ and $B$ that the product of $p$-adic densities in the main term becomes $\prod_{p\notin M}\sigma_{p}\cdot\prod_{p\in M }\sigma^{B,\r}_{p}$. Thus, this yields the asymptotic formula for $N(A,\mathfrak{B},B,\r,P)$ as desired in Lemma $\ref{lem1.31.31.3}$. For a detailed proof, see also [$\ref{ref1555}$, Lemma 4.4].

\begin{lem}\label{lem2.5}
     Suppose that $A$ and $Q$ are positive numbers with $Q\leq A$.  Then, for a given $\c\in \Z^N$ with $1\leq \c\leq Q$, we have
    $$\#\left\{\a\in [-A,A]^N\cap \Z^N\middle|\ P(\a)=0\ \text{and}\ \a\equiv \c\ (\text{mod}\ Q)\right\}\ll (A/Q)^{N-k},$$
    where the implicit constant may depend on $P(\mathbf{t}).$
\end{lem}
\begin{proof}
    By orthogonality, we have
    \begin{equation}\label{ortho2.17}
        \begin{aligned}
          &  \#\left\{\a\in [-A,A]^N\cap \Z^N\middle|\ P(\a)=0\ \text{and}\ \a\equiv \c\ (\text{mod}\ Q)\right\}\\
           & =\dint_0^1\dsum_{-A\leq Q\y+\c\leq A}e(\alpha P(Q\y+\c))d\alpha.
        \end{aligned}
    \end{equation}
    By change of variable $\alpha=\beta /Q^k,$ the last expression is seen to be 
    \begin{equation}\label{exp2.17}
    \begin{aligned}
      &  Q^{-k}\dint_0^{Q^k} \dsum_{-A\leq Q\y+\c\leq A}e(\beta P(\y+\c/Q))d\beta\\
       &\ll \sup_{\substack{1\leq l\leq Q^k\\ l\in \N}}\dint_{l-1}^l \dsum_{-A\leq Q\y+\c\leq A}e(\beta P(\y+\c/Q))d\beta\\
       &=\sup_{\substack{1\leq l\leq Q^k\\ l\in \N}}\dint_{l-1}^l \dsum_{-A\leq Q\y+\c\leq A}e(\beta( P(\y)+ g(\y)))d\beta,
    \end{aligned}
    \end{equation}
    where $g\in \Q[\y]$ is a polynomial of degree at most $k-1.$

On writing that $S(\beta)=\sum_{-A\leq Q\y+\c\leq A}e(\beta(P(\y)+g(\y))),$ we claim that whenever $N>(k-1)2^k$, one has
\begin{equation}\label{2.32.3}
    \dint_{l-1}^l S(\beta) d\beta \ll (A/Q)^{N-k},
\end{equation}
uniformly in $l.$ Suppose that the inequality $(\ref{2.32.3})$ holds. Then, on substituting $(\ref{2.32.3})$ into the last expression of $(\ref{exp2.17})$, it follows from $(\ref{ortho2.17})$ and $(\ref{exp2.17})$ that we complete the proof of Lemma 2.2.

It remains to verify the inequality $(\ref{2.32.3})$.   For a given $l\in \Z,$ we define the major arcs $\mathfrak{M}^l(H)$ by
$$\mathfrak{M}^l(H)=\bigcup_{\substack{0\leq a\leq q\leq H\\ (q,a)=1}}\mathfrak{M}^l(H,q,a),$$
    where 
    $$\mathfrak{M}^l(H,q,a)=\left\{\beta\in [l-1,l):\ \left|\beta-(l-1)-\frac{a}{q}\right|\leq \frac{H}{q(A/Q)^{k}}\right\}.$$
    Furthermore, we define the minor arcs $\mathfrak{m}^l:=[l-1,l)\setminus \mathfrak{M}^l.$
    
For simplicity, we put $L=A/Q$. Let $\delta$ be a positive number with $\delta<2^{-1-k}$. Define $I\in\N$ to be the minimum number satisfying $2^IL^{\delta}>(1/4)L^{k/2}$. We notice here that $I=O_{k,\delta}(\log L).$ By using those major and minor arcs dissections of $[l-1,l)$, defined in the previous paragraph, we deduce that
    \begin{equation}\label{goal}
        \dint_{l-1}^l S(\beta)d\beta \ll S_1+S_2+S_3,
    \end{equation}
    where \begin{equation*}
        \begin{aligned}
            S_1&=\dint_{\mathfrak{M}^l(L^{\delta})}|S(\beta)|d\beta\\
            S_2&=\dsum_{i=1}^I\dint_{\mathfrak{M}^l(2^iL^{\delta})\setminus \mathfrak{M}^l(2^{i-1}L^{\delta})}|S(\beta)|d\beta\\
            S_3&=\dint_{\mathfrak{m}^l((1/4)L^{k/2})}|S(\beta)| d\beta.
        \end{aligned}
    \end{equation*}

 We shall show that whenever $N>(k-1)2^k$, each of $S_1,S_2$ and $S_3$ is $O(L^{N-k}),$ which delivers the desired bound $(\ref{2.32.3}).$  We first obtain the upper bound for $S_3.$ On noting that 
 \begin{equation*}
 \begin{aligned}
     S(\beta)&=\dsum_{-A\leq Q\y+\c\leq A}e(\beta (P(\y)+g(\y)))\\
     &=\dsum_{-A\leq Q\y+\c\leq A}e((\beta-(l-1))P(\y)+\beta g(\y)),
 \end{aligned}
 \end{equation*}
and recalling that $P(\t)$ is a non-singular form, we infer by [$\ref{ref100.1}$, Lemma 3.6] with $R=1$, $\alpha_1=\beta-(l-1)$, $f_1(\x)=P(\y)$, $G(\x)=\beta g(\y),\text{dim} V_{f_1}^*=0,\kappa=N$ that 
\begin{equation*}
    \sup_{\beta\in \mathfrak{m}^l((1/4)L^{k/2})} |S(\beta)|\ll L^{N-kN/((k-1)2^{k})+\epsilon}.
\end{equation*}
Hence, whenever $N>(k-1)2^k$, one has 
\begin{equation}
    \begin{aligned}
       S_3&\ll \sup_{\beta\in \mathfrak{m}^l((1/4)L^{k/2})}|S(\beta)|\cdot \dint_{l-1}^l 1d\beta\ll L^{N-k}.
    \end{aligned}
\end{equation}

Next, we derive the upper bound for $S_2.$ Note that $\text{mes}(\mathfrak{M}^l(H))\ll H^2L^{-k}$ and note again by [$\ref{ref100.1}$, Lemma 3.6] that whenever $N> (k-1)2^k$ one has
\begin{equation*}
\begin{aligned}
    \sup_{\beta\in \mathfrak{m}^l(2^{i-1}L^{\delta})}|S(\beta)| &\ll L^{N}(2^{i-1}L^{\delta})^{-N/((k-1)2^{k-1})+\epsilon}\\
    &\ll L^N(2^{i-1}L^{\delta})^{-2+\epsilon}\cdot (2^{i-1}L^{\delta})^{-1/((k-1)2^{k-1})}.
\end{aligned}
\end{equation*}
Hence, one has
\begin{equation*}
    \begin{aligned}
        \dint_{\mathfrak{M}^l(2^iL^{\delta})\setminus \mathfrak{M}^l(2^{i-1}L^{\delta})}|S(\beta)|d\beta&\ll \text{mes}(\mathfrak{M}^l(2^iL^{\delta}))\cdot \sup_{\beta\in \mathfrak{m}^l(2^{i-1}L^{\delta})}|S(\beta)|\\
        &\ll L^{N-k}\cdot (2^iL^{\delta})^{\epsilon-1/((k-1)2^{k-1})}\\
        & \ll L^{N-k-\eta},
    \end{aligned}
\end{equation*}
for all $i=1,2,\ldots, I$ and with some $\eta>0.$ Therefore, we find by $I=O(\log L)$ that $S_2=O(L^{N-k}).$
 
 %   Then, on writing that 
  %  \begin{equation}\label{S2.20}
   % S(\beta):=\dsum_{-A\leq Q\y+\c\leq A}e(\beta( P(\y)+ g(\y))),
   % \end{equation}
   % we see that
   % \begin{equation*}
   %     \dint_{l-1}^{l}S(\beta)d\beta=\dint_{\mathfrak{M}^l_{\delta}}S(\beta)d\beta+\dint_{\mathfrak{m}^l_{\delta}}S(\beta)d\beta.
   % \end{equation*}

%Note  and the fact that $P(\mathbf{t})$ is a non-singular form that
%\begin{equation}\label{ineq}
 %   \sup_{\beta\in \mathfrak{m}^l_{\delta}}S(\beta)\ll (A/Q)^{N-N\delta/(2^{k-1}(k-1))+\epsilon}.
%\end{equation}
 % One infers by applying the argument in [$\ref{ref8}$] with the bound $(\ref{ineq})$ that whenever $N>(k-1)2^{k}$ and $A/Q>0$ is sufficiently large, we have
  % \begin{equation}\label{bound2.20}
   % \dint_{l-1}^{l}S(\beta)d\beta\ll (A/Q)^{N-k}.
    %\end{equation}
  % We remark here that when we follow the argument in [$\ref{ref8}$] to obtain $(\ref{bound2.20})$, we first need to use Weyl differencing procedure to get rid of the effect of the function $g(\y)$, and we next follow the classical major and minor arcs treatments appeared in [$\ref{ref8}$] (see also [$\ref{ref1000}$, Lemma 2.10]).

   Lastly, we deduce the upper bound for $S_1.$ For this, it is convenient to define differencing operators $\Delta_1$ by
\begin{equation*}
    \Delta_1(P(\x);\h)=P(\x+\h)-P(\x),
\end{equation*}
and so we define $\Delta_j$ for $j\geq 2$ recursively by means of the relations
\begin{equation}\label{3.3}
    \Delta_j(P(\x);\h_1,\ldots,\h_{j})=\Delta_1(\Delta_{j-1}(P(\x);\h_1,\ldots,\h_{j-1});\h_j).
\end{equation}

For given variables $\y^{(1)},\ldots,\y^{(k)}\in \Z^N$, we define 
\begin{equation*}
    \begin{aligned}
       & \phi_1(\y^{(k-1)},\y^{(k)})\\
        &:=\phi_1(\y^{(k-1)},\y^{(k)};\y^{(1)},\ldots,\y^{(k-2)})\\
        &=\Delta_{k-2}(P(\y^{(k-1)});\y^{(1)},\ldots,\y^{(k-2)})-\Delta_{k-2}(P(\y^{(k)});\y^{(1)},\ldots,\y^{(k-2)}).
    \end{aligned}
\end{equation*}
Then, by applying the Cauchy-Schwarz inequality together with the classical Weyl differencing argument, we deduce that 
\begin{equation}\label{2.62.6}
    S(\beta)^{2^{k-1}}\ll L^{(2^{k-1}-k)N}\dsum_{\y^{(1)},\ldots,\y^{(k-2)}}\dsum_{\y^{(k-1)},\y^{(k)}\in \mathcal{D}L}e(\beta \phi_1(\y^{(k-1)},\y^{(k)})),
\end{equation}
where the variables $\y^{(1)},\ldots,\y^{(k-2)}$ run over $[-2L,2L]^N$ and $$\mathcal{D}:=\mathcal{D}(\y^{(1)},\ldots,\y^{(k-2)})$$ is a box in $[-2,2]^N$ suitably defined through the classical Weyl differencing argument in terms of $\y^{(1)},\ldots,\y^{(k-2)}.$

We denote by $\mathcal{Y}(\beta;\mathcal{D})$ the exponential sum over  $\y^{(1)},\ldots,\y^{(k)}$ on the right hand side of $(\ref{2.62.6})$. Obviously, we have
\begin{equation}\label{2727}
    \mathcal{Y}(\beta;\mathcal{D})\ll \sup_{\mathcal{B}\subseteq [-1,1]^N}|\mathcal{Y}(\beta;\mathcal{B})|,
\end{equation}
where $\mathcal{B}$ is over boxes in $[-2,2]^N$. For $\beta\in \mathfrak{M}^l(L^{\delta},q,a)$ and for a given box $\mathcal{B}\subseteq[-2,2]^N$, we derive the upper bound for $\mathcal{Y}(\beta;\mathcal{B})$. To do this, we make use of the classical argument of Birch [$\ref{ref8}$, Section 5].  Hence, we shall be brief in some steps.

%Define
%\begin{equation*}
 %   S(q,a)=\dsum_{1\leq \z^{(1)},\ldots,\z^{(k)}\leq q}e\bigl(\frac{a}{q}\phi_1(\z^{(k-1)},\z^{(k)}; \z^{(1)},\ldots,\z^{(k-2)})\bigr),
%\end{equation*}
%where $\z^{(i)}$ $(i=1,\ldots,k)$ runs over $[1,q]^N.$ % Also, for $\alpha\in \R$ and for given $\z^{(1)},\ldots,\z^{(k)}\in \Z^N,$ we define 
%\begin{equation}
 %   \mathcal{T}(\alpha, \z^{(1)},\ldots,\z^{(k)})=\dsum_{\x^{(1)},\ldots,\x^{(k)}}e(\alpha \phi_2(\x^{(k-1)},\x^{(k)})),
%\end{equation}

Put $\alpha=\beta-(l-1)-a/q.$ Write $\y^{(i)}=q\x^{(i)}+\z^{(i)}\ (i=1,\ldots,k),$ where $1\leq \z^{(i)}\leq q$ and $\x^{(i)}$ runs over boxes so that $\|\y\|_{\infty}\leq 2L$ with $i=1,\ldots,k-2,$ and $\x^{(k-1)},\x^{(k)}$ run over boxes so that $\y^{(k-1)},\y^{(k)}\in \mathcal{B}L$. Then, we have
\begin{equation}\label{2.72.7}
    \mathcal{Y}(\beta; \mathcal{B})= \dsum_{1\leq \z^{(1)},\ldots,\z^{(k)}\leq q}e\bigl(\frac{a}{q}\phi_1(\z^{(k-1)},\z^{(k)}; \z^{(1)},\ldots,\z^{(k-2)})\bigr)\mathcal{T}(\alpha, \z^{(1)},\ldots,\z^{(k)}),
\end{equation}
where 
$$\mathcal{T}(\alpha, \z^{(1)},\ldots,\z^{(k)})=\dsum_{\x^{(1)},\ldots,\x^{(k)}}e(\alpha \phi_2(\x^{(k-1)},\x^{(k)})),$$
in which 
\begin{equation*}
\begin{aligned}
    &\phi_2(\x^{(k-1)},\x^{(k)})\\
    &:= \phi_2(\x^{(k-1)},\x^{(k)};\x^{(1)},\ldots,\x^{(k-2)})\\
    &=\phi_1(q\x^{(k-1)}+\z^{(k-1)},q\x^{(k)}+\z^{(k)};q\x^{(1)}+\z^{(1)},\ldots,q\x^{(k-2)}+\z^{(k-2)}).
\end{aligned}
\end{equation*}

Note that $\sup_{\boldsymbol{\gamma}\in [0,1]^N}|\boldsymbol{\gamma}\cdot \nabla e(\alpha\phi_2(\x^{(k-1)},\x^{(k)}))|\ll qL^{k-1}|\alpha|$, since $\phi_2(\x^{(k-1)},\x^{(k)}))$ is a polynomial of degree $k$ in $\x^{(1)},\ldots,\x^{(k)}$. Hence, if we define 
\begin{equation*}
\begin{aligned}
    I(\alpha;\mathcal{B})=\dint_{[-2,2]^{(k-2)N}}\dint_{\mathcal{B}^2}e(\alpha\phi_1(\boldsymbol{\eta}^{(k-1)},\boldsymbol{\eta}^{(k)};\boldsymbol{\eta}^{(1)}\ldots,\boldsymbol{\eta}^{(k-2)}))d\boldsymbol{\eta},
\end{aligned}
\end{equation*}
where $d\boldsymbol{\eta}=d\boldsymbol{\eta}^{(k-1)}d\boldsymbol{\eta}^{(k)}d\boldsymbol{\eta}^{(1)}\cdots d \boldsymbol{\eta}^{(k-2)},$ it follows by the classical argument in [$\ref{ref8}$, Section 5] using multi-dimensional mean value theorem that 
\begin{equation}\label{2.82.8}
    \mathcal{T}(\alpha, \z^{(1)},\ldots,\z^{(k)})-(L/q)^{kN}I(L^k\alpha)\ll E_1+E_2,
\end{equation}
where
\begin{equation*}
\begin{aligned}
     &E_1=q(L/q)^{kN}L^{k-1}|\alpha|\\
      &E_2=(L/q)^{kN-1}.
\end{aligned}
\end{equation*}

Write 
\begin{equation}\label{SS}
    S(q,a)=\dsum_{1\leq \z^{(1)},\ldots,\z^{(k)}\leq q}e\bigl(\frac{a}{q}\phi_1(\z^{(k-1)},\z^{(k)}; \z^{(1)},\ldots,\z^{(k-2)})\bigr).
\end{equation} Then, on substituting $(\ref{2.82.8})$ into $(\ref{2.72.7})$, we obtain that
\begin{equation}\label{210210}
    \mathcal{Y}(\beta;\mathcal{B})=q^{-kN}S(q,a)I(L^k\alpha;\mathcal{B})L^{kN}+O(E(\beta)),
\end{equation}
 where $E(\beta)=q^{kN}(E_1+E_2).$ By the definition $E_1$ and $E_2$ together with the fact that $\beta\in \mathfrak{M}^l(L^{\delta},q,a)$, we deduce that
 \begin{equation}\label{2.11}
     E(\beta)\ll L^{kN-1+\delta}.
 \end{equation}

Meanwhile, it follows from $(\ref{2.62.6})$ together with $(\ref{2727})$ that
\begin{equation}\label{2.122.12}
    S_1\ll L^{(1-2^{1-k}k)N}\dint_{\mathfrak{M}^l(L^{\delta})}\sup_{\mathcal{B}\subseteq [-1,1]^N}|\mathcal{Y}(\beta;\mathcal{B})|^{2^{1-k}}d\beta.
\end{equation}
Let us write 
 \begin{equation}\label{sigma}
     \mathfrak{S}(L^{\delta})=\dsum_{1\leq q\leq L^{\delta}}\dsum_{\substack{1\leq a\leq q\\(q,a)=1}}(q^{-kN}S(q,a))^{2^{1-k}}.
 \end{equation}
 Then, on noting that $\text{mes}(\mathfrak{M}^l(L^{\delta}))\ll L^{2\delta-k}$ and substituting $(\ref{2.11})$ into $(\ref{210210})$ and that into $(\ref{2.122.12})$, we find that
 \begin{equation}\label{S_1}
 \begin{aligned}
       S_1\ll L^{N-k} \cdot\mathfrak{S}(L^{\delta})\cdot\dint_{|\alpha_1|\leq \frac{L^{\delta}}{q}}\sup_{\mathcal{B}\subseteq [-1,1]^N}|I(\alpha_1;\mathcal{B})|^{2^{1-k}}d\alpha_1
       +L^{N-k+2\delta-(1-\delta)2^{1-k}},
 \end{aligned}
 \end{equation}
where we have assumed that $\sup_{\mathcal{B}\subseteq [-1,1]^N}|\mathcal{Y}(\beta;\mathcal{B})|^{2^{1-k}}$ and $\sup_{\mathcal{B}\subseteq [-1,1]^N}|I(L^k\alpha;\mathcal{B})|^{2^{1-k}}$ is a measurable function and we have used a change of variable $\alpha_1=L^k\alpha.$ We will prove the measurability for these functions at the end of the proof of this lemma.

As the endgame, we shall use [$\ref{ref100.1}$, Lemma 3.4]  to obtain the upper bound the functions $\mathfrak{S}(L^{\delta})$ and $I(\alpha_1:\mathcal{B})$. Recall the definition of $\mathcal{Y}(\beta;\mathcal{B})$. By applying the Weyl's differencing argument, we see that
\begin{equation}\label{2.13}
\begin{aligned}
\mathcal{Y}(\beta;\mathcal{B})\leq \dsum_{\y^{(1)},\ldots,\y^{(k-1)}}\bigl|\dsum_{\y^{(k)}\in \mathcal{B}_1L}e(\beta \Delta_{k-1}(P(\y^{(k)});\y^{(1)},\ldots,\y^{(k-1)}))\bigr|,
\end{aligned}
\end{equation}
where $ \mathcal{B}_1\subseteq[-2,2]^N$ is a box suitably defined through the classical Weyl differencing argument in terms of $\y^{(k-1)}.$ By [$\ref{ref100.1}$, Lemma 3.4] with $P=L,R=1, \alpha_1=\beta-(l-1), f_1=P(\y), \text{dim} V^*_{f_1}=0$ and $\kappa=N$, we infer that whenever $\beta\notin \mathfrak{M}^l(H)$ with $H\leq L^{k-1}$, one has
\begin{equation}\label{2.142.14}
\begin{aligned}
    \mathcal{Y}(\beta;\mathcal{B})= \mathcal{Y}(\beta-(l-1);\mathcal{B})\ll L^{kN+\epsilon}H^{-N/(k-1)}.
    \end{aligned}
\end{equation}

Meanwhile, we note that  if we temporarily assume that $|\beta-(l-1)|<(1/2)L^{-k/2},$ it follows that $\mathfrak{M}^l(|\beta-(l-1)| L^{k},q,a)$ are disjoint over $0\leq a\leq q\leq |\beta-(l-1)| L^{k}$ with $(q,a)=1$, and $\beta$ is on the edge of  $\mathfrak{M}^l(|\beta-(l-1)| L^{k},1,0)$. Thus, one sees that $\beta\notin \mathfrak{M}^l(H)$ with $H=|\beta-(l-1)|L^{k-\epsilon}$ for any $\epsilon>0$.
Hence, it follows by $(\ref{2.142.14})$ that whenever $|\beta-(l-1)|<(1/2)L^{-k/2},$ one has
\begin{equation}\label{1515}
    \mathcal{Y}(\beta;\mathcal{B})\ll L^{kN+\epsilon}(|\beta-(l-1)|L^k)^{-N/(k-1)}.
\end{equation}

First, we claim that for any $\alpha_1>0$ one has
\begin{equation}\label{216216}
    I(\alpha_1;\mathcal{B})\ll \min\{1, |\alpha_1|^{-N/(k-1)+\epsilon}\}.
\end{equation}
The argument for this is based on [$\ref{ref8}$, Lemma 5.2]. The estimate $I(\alpha_1;\mathcal{B})\ll 1$ is trivial. To obtain the second bound, we may assume that $|\alpha_1|>1.$ Taking $a=0$ and $q=1$ in $(\ref{210210}),$ we deduce from the definition of $E_1$ and $E_2$ that 
\begin{equation}\label{217217}
    \mathcal{Y}(\beta;\mathcal{B})=I(L^k\alpha;\mathcal{B})L^{kN}+O((|\alpha|L^{k}+1)L^{kN-1}),
\end{equation}
with $\alpha=\beta-(l-1).$
 Then, on writing $L^k\alpha=\alpha_1$, it follows  by $(\ref{217217})$ together with $(\ref{1515})$ that whenever $1<|\alpha_1|<(1/2)L^{k/2}$, one has
\begin{equation}\label{218218}
    I(\alpha_1;\mathcal{B})\ll |\alpha_1|^{-N/(k-1)}L^{\epsilon}+ |\alpha_1|L^{-1}.
\end{equation}
On observing that $I(\alpha_1;\mathcal{B})$ does not depend on $L,$ by taking $L=|\alpha_1|^{1+N/(k-1)}$, the inequality $(\ref{218218})$ delivers $(\ref{216216}).$

Recall the definition ($\ref{SS}$) of $S(q,a).$ Next, we claim that whenever $(q,a)=1$, one has
\begin{equation}\label{SSS}
    S(q,a)\ll q^{kN-N/(k-1)+\epsilon}.
\end{equation}
By applying the Weyl's differencing argument, we see that 
\begin{equation}
    S(q,a)\leq \dsum_{1\leq \z^{(1)},\ldots,\z^{(k-1)}\leq q}\bigl|\dsum_{1\leq \z^{(k)}\leq q}e\bigl(\frac{a}{q}\Delta_{k-1}(P(\z^{(k)});\z^{(1)},\ldots,\z^{(k-1)}\bigr)\bigr|.
\end{equation}
 Meanwhile, note that for any $\epsilon>0,$ one cannot find $q'\in \N,a'\in \Z$ with $(q',a')=1$ and $1\leq q'\leq q^{1-\epsilon}$ satisfying 
  \begin{equation}
      \begin{aligned}
          |q'a-qa'|\leq q^{1-\epsilon}q^{1-k}
      \end{aligned}
  \end{equation}
Hence,  by applying again [$\ref{ref100.1}$, Lemma 3.4] with $P=q,\ R=1,\ \alpha_1=a/q ,\ f_1=P(\y),\ \text{dim} V^*_{f_1}=0,\ X=q^{(1-\epsilon)/(k-1)}$ and $\kappa=N$, we infer that whenever $(q,a)=1$, one sees that the inequality $(\ref{SSS})$ holds.

On substituting ($\ref{SSS}$) into $(\ref{sigma})$, we see that whenever $N>(k-1)2^{k}$
\begin{equation}
    \mathfrak{S}(L^{\delta})\ll 1.
\end{equation}
Hence, on substituting ($\ref{216216}$) into $(\ref{S_1})$, one infers that $S_1=O(L^{N-k}),$ since $0<\delta<2^{-1-k}.$  We have shown thus far that $S_1+S_2+S_3=O(L^{N-k}),$ which establishes $(\ref{2.32.3})$ by $(\ref{goal})$.

We complete the proof of this lemma by confirming that the functions $g_1(\beta):=\sup_{\mathcal{B}\subseteq [-1,1]^N}|\mathcal{Y}(\beta;\mathcal{B})|^{2^{1-k}}$ and $g_2(\alpha):=\sup_{\mathcal{B}\subseteq [-1,1]^N}|I(L^k\alpha;\mathcal{B})|^{2^{1-k}}$ are measurable. It is enough to show that for any given $k>0,$ the set $\{\alpha\in \R: g_i(\alpha)>k\}$ ($i=1,2$)
is open. Let $g_2(\alpha_0)>k$. Then, there exists $\mathcal{B}\subseteq [-1,1]^N$ such that $|I(L^k\alpha;\mathcal{B})|^{2^{1-k}}>k$. On recalling the definition of $I(\cdot;\mathcal{B} )$, one sees that there exists $\epsilon>0$  such that whenever $|\alpha-\alpha_0|<\epsilon$ we have $|I(L^k\alpha;\mathcal{B})|^{2^{1-k}}>k$, which proves that the set 
$\{\alpha\in \R: g_2(\alpha)> k\}$
is open as desired. By applying the same argument, we see that $g_1(\beta)$ is a measurable function.
   
  % Recall that $$ \#\left\{\a\in [-A,A]^N\cap \Z^N\middle|\ P(\a)=0\ \text{and}\ \a\equiv \c\ (\text{mod}\ Q)\right\}$$ is bounded above by the last expression in $(\ref{exp2.17})$ and recall the definition $(\ref{S2.20})$ of $S(\beta)$. Therefore, by substituting the bound $(\ref{bound2.20})$ into the last expression of $(\ref{exp2.17})$, we conclude that  $$\#\left\{\a\in [-A,A]^N\cap \Z^N\middle|\ P(\a)=0\ \text{and}\ \a\equiv \c\ (\text{mod}\ Q)\right\}\ll (A/Q)^{N-k}.$$
\end{proof}

%\bigskip

The following lemma provides an upper asymptotic estimate for the number of integer points in the variety cut by the polynomial $P(\mathbf{t})\in \Z[\mathbf{t}]$, that reduce modulo $p$, for some sufficiently large $p>M,$ to an $\mathbb{F}_{p}$-point of an another given variety $Y\subset \mathbb{A}^n$ defined over $\Z$.

\begin{lem}\label{lem2.42.4}
Let $\mathfrak{B}$ be a compact region in $\R^N$ having a finite measure, and let $Y$ be any closed subscheme of $\mathbb{A}^N_{\Z}$ of codimension $r\geq 1.$ Let $A$ and $M$ be positive real numbers. Suppose that  $r-1>k$. Then, there exists $A_0:=A_0(P(\mathbf{t}))\in \R_{>0}$ such that whenever $A>A_0$, %\textcolor{purple}{HL: this means $A>M$ for some $M$ depending on $P(\mathbf{t})$? KY: I will rewrite this more precisely, thank you for pointing out!}
we have
    \begin{equation}\label{ineq2.162.16}
    \begin{aligned}
     &\#\left\{\a\in A\mathfrak{B}\cap \Z^N\middle|\ \begin{aligned}
        &(i) \ \a\ (\text{mod}\ p)\in Y(\mathbb{F}_p)\ \text{for some prime}\ p>M\\
        &(ii)\ P(\a)=0
     \end{aligned}\right\}\\
     &\ll\frac{A^{N-k}}{M^{r-k-1}\log M}+A^{N-r+1},
    \end{aligned}
    \end{equation}
    where the implicit constant may depend on $\mathfrak{B}$ and $Y$.
\end{lem}

In [$\ref{ref13}$, Theorem 3.3], Bhargava provided an upper asymptotic estimate that 
    \begin{equation}\label{ineq2.16}
    \begin{aligned}
     &\#\left\{\a\in A\mathfrak{B}\cap \Z^N\middle|
       \ \a\ (\text{mod}\ p)\in Y(\mathbb{F}_p)\ \text{for some prime}\ p>M\right\}\\
    & \ll\frac{A^{N}}{M^{r-1}\log M}+A^{N-r+1}.
    \end{aligned}
    \end{equation}
 Furthermore, as alluded in [$\ref{ref13}$, Remark 3.4], the bound in $(\ref{ineq2.16})$ can be achieved for suitable choices of $Y$. Thus, this bound is essentially optimal.

In the proof of Lemma $\ref{lem2.42.4}$, we mainly adopt the argument in [$\ref{ref13}$, Theorem 3.3] though, we freely admit that the bound in $(\ref{ineq2.162.16})$ seems not optimal order of magnitude of the bound. Especially, one finds that the second term is trivially obtained by that in $(\ref{ineq2.16})$. We are independently interested in a sharper bound in ($\ref{ineq2.162.16}$) and expect that one might be able to improve this bound. Nevertheless, since the strength of the upper asymptotic estimate in Lemma $\ref{lem2.42.4}$ is enough for our purpose, we do not put our effort into optimizing this upper bound in this paper.

\begin{proof}[Proof of Lemma \ref{lem2.42.4}]
We can and do assume that $Y$ is irreducible. Otherwise, we can take its irreducible components and add up the equation \eqref{ineq2.162.16} to deduce general cases. Since $Y$ has codimension $r$, there exists $f_1,\dots,f_r \in \Z[t_1,\dots,t_N]$ such that the vanishing locus $V(f_1,\dots,f_r)$ contains an irreducible component of codimension $r$ containing $Y$. Indeed, we can assume that $Y$ equals to the irreducible component, as they have the same underlying reduced subscheme, and we only consider $\Z$ or $\mathbb{F}_p$-points of them.  By [$\ref{ref13}$, Lemma 3.1], the number of $\a \in A\mathfrak{B}\cap Y(\Z)$ is $\ll A^{N-r}$. (Note that $A\mathfrak{B}\cap Y(\Z)$ equals to $A\mathfrak{B}\cap \Z^N\cap Y(\R)$) Thus, it suffices to find an upper bound of the size of the set by
\begin{equation*}
\begin{aligned}
   & \#\left\{\a\in A\mathfrak{B}\cap \Z^N\middle|\ \begin{aligned}
        &(i) \ P(\a)=0\\
        &(ii)\ \a\ (\text{mod}\ p)\in Y(\mathbb{F}_p)\ \text{for some prime}\ p>M \\
        &(iii) \ \a \notin Y(\Z)
     \end{aligned}\right\} \\
    & \ll \frac{A^{N-k}}{M^{r-k-1}\log M }+ A^{N-r+1}.
     \end{aligned}
\end{equation*}
We may assume that $r>k+1$, since $r-1>k$ from the hypothesis in the statement of Lemma $\ref{lem2.42.4}$. %For $r=1$, this follows easily because ... \textcolor{purple}{HL: Bhargava says this is obvious, but is it?? KY: Because of the second term in the right-hand side, it becomes trivial!} 
We shall find an upper bound of a slightly larger set by
\begin{equation}\label{ineq2.18}
\begin{aligned}
     & \#\left\{(\a,p) \middle|  \begin{aligned}
        &(i) \  \a \in A\mathfrak{B}\cap \Z^N \  \text{and} \ P(\a)=0\\
        &(ii)\ p>M \  \text{a prime and} \  \a\ (\text{mod}\ p)\in Y(\mathbb{F}_p) \\
        &(iii) \ \a \notin Y(\Z)
     \end{aligned}\right\} \\
    & \ll \frac{A^{N-k}}{M^{r-k-1}\log M }+ A^{N-r+1}.
\end{aligned}
\end{equation}
%\textcolor{purple}{HL: (1) $Y'$ in $\ref{ref13}$ can be identified with $Y$ and that's why $Y'$ does not appear here. (2) I used $\ll$ instead of the big O notation. I suppose that they mean the same?? KY: Yes, they are same}

First, we count pairs ($\a,p$) on the left hand side of $(\ref{ineq2.18})$ for each prime $p$ satisfying $p\leq A$; such primes arise only when $A>M.$ Meanwhile, by Lemma $\ref{lem2.5}$ we find that for a given $\c\in [1,p]^N$, the number of integer solutions $\a\in [-A,A]^N$ of $P(\a)=0$ with the congruence condition $\a\equiv \c\ (\text{mod}\ p)$ is $O((A/p)^{N-k}).$ Then, since $\#Y(\mathbb{F}_p)=O(p^{N-r}),$ we see that the number of $\a\in [-A,A]^N\cap \Z^N$ such that $\a\ (\text{mod}\ p)$ is in $Y(\mathbb{F}_p)$ is $O(p^{N-r})\cdot O((A/p)^{N-k})=O(A^{N-k}/p^{r-k}).$ Thus the total number of desired pairs $(\a,p)$ with $p\leq A$ is at most
\begin{equation}\label{ineq2.19}
\begin{aligned}
     \#&\left\{(\a,p) \middle|  \begin{aligned}
        &(i) \  \a \in A\mathfrak{B}\cap \Z^N \  \text{and} \ P(\a)=0\\
        &(ii)\ A\geq p>M \  \text{a prime and} \  \a\ (\text{mod}\ p)\in Y(\mathbb{F}_p) \\
        &(iii) \ \a \notin Y(\Z)
     \end{aligned}\right\} \\
    & \ll \dsum_{M<p\leq A}O\left(\frac{A^{N-k}}{p^{r-k}}\right)=O\left(\frac{A^{N-k}}{M^{r-k-1}\log M}\right).
\end{aligned}
\end{equation}

Next, we count pairs $(\a,p)$ with $p>A.$ It follows from (see the equation $(17)$ in the proof of  $[\ref{ref13}, \text{Theorem}\ 3.3]$) that 
\begin{equation}\label{ineq2.20}
\begin{aligned}
     &\#\left\{(\a,p) \middle|  \begin{aligned}
        &(i) \  \a \in A\mathfrak{B}\cap \Z^N \  \text{and} \ P(\a)=0\\
        &(ii)\ p>A \  \text{a prime and} \  \a\ (\text{mod}\ p)\in Y(\mathbb{F}_p) \\
        &(iii) \ \a \notin Y(\Z)
     \end{aligned}\right\} \\
  \ll &\   \#\left\{(\a,p) \middle|  \begin{aligned}
        &(i) \  \a \in A\mathfrak{B}\cap \Z^N\\
        &(ii)\ p>A \  \text{a prime and} \  \a\ (\text{mod}\ p)\in Y(\mathbb{F}_p) \\
        &(iii) \ \a \notin Y(\Z)
     \end{aligned}\right\}\ll  A^{N-r+1}.
\end{aligned}
\end{equation}
Therefore, we find by $(\ref{ineq2.19})$ and $(\ref{ineq2.20})$ that the inequality ($\ref{ineq2.18}$) holds. Hence, we complete the proof of Lemma $\ref{lem2.42.4}.$\qedhere

\end{proof}

%\bigskip

%Consider a set $S(A)_{p_0}$ consisting of all elements $\a\in [-A,A]\cap \Z^N $ such that $P(\a)=0$ for some $P\in \Z[t_1, \dots, t_N]$ and $\langle \a , v_{d,n}(\x) \rangle = 0$ has a solution in $\R$ and $\Q_p$ for all $p\leq p_0$. 

Next, we will prove that most of $f_{\a}(\x)$ are irreducible. Let $Y$ be a subset of $\A_\Z^N$ defined to be \begin{equation}\label{Eqn.defnofY} Y:=\left\{\a\in \A_\Z^N\middle|\ f_{\a}(\x) \text{ is reducible over}\ \C \right\}.\end{equation}  Our goal here is to show that $Y$ is, in fact, an algebraic variety and that the codimension of $Y$ is strictly greater than a constant depending on $n$ and $d$. Here, we fix $r:=\codim_{\bA_\Z^N} Y$. To show the claim, we record the following two useful lemmas:

%\todo{Seungsu: Write the setup for the theorem 2.6 here. Definition of $Y$ and $r$ in particular. Correct everything below to match the notation.}

\begin{lem}\label{lem.factorialcomparison}
For any integers $n\geq 3$ and $d\geq 3$, 
\[\frac{(n+d)(n+d-1)}{n-1} < {n+d \choose d}\]
\end{lem}

\begin{proof}
We do this by induction on $n$. When $n=3$, we have

\[\begin{split}
3 &< 3+d-2 \\
\frac{3(3+d)(3+d-1)}{2} &<\frac{(3+d)(3+d-1)(3+d-2)}{2}  \\
\frac{(3+d)(3+d-1)}{2} &< \frac{(3+d)(3+d-1)(3+d-2)}{3\cdot 2}={n+d \choose d}
\end{split}\]

Suppose that the lemma is true for $n$. Then,
\begin{equation}\label{eq.InductionComputation}
\begin{split}
\frac{(n+d+1)(n+d)}{n} & = \frac{(n+d+1)(n+d)(n+d-1)(n-1)}{n(n-1)(n+d-1)}\\
& <\frac{(n+d)!}{n! d!}\cdot\frac{(n+d+1)(n-1)}{n(n+d-1)}\\
& = \frac{(n+d+1)!}{(n+1)!d!}\cdot \frac{(n+1)(n-1)}{n(n+d-1)}\\
& < \frac{(n+d+1)!}{(n+1)!d!}
\end{split}
\end{equation}
The last inequality follows from $\frac{(n-1)(n+1)}{n(n+d-1)}<1$.
\end{proof}

\begin{lem}\label{lem.Cndmaximum}
Let $n\geq 3$ and $d \geq 3$ be integers. Suppose that $d_1$ and $d_2$ are natural numbers with $d=d_1 + d_2$. Let $C_{n,d}(d_1, d_2) $ be a rational number defined as \[C_{n,d}(d_1, d_2) =\frac{d!(n+d_1 - 1)!}{d_1!(n+d-1)!}+\frac{d!(n+d_2 - 1)!}{d_2!(n+d-1)!}\] Then, for given $n$ and $d,$ the quantity $C_{n,d}$ attains the maximum value when $(d_1, d_2) = (1, d-1)$ or $(d_1, d_2) = (d-1, 1)$. Furthermore, its maximum is strictly less than 1.
\end{lem}

\begin{proof}
Without loss of generality, we assume that $d_1\leq d_2.$ We shall first show that one has
\begin{equation}\label{1.8}
    C_{n,d}(d_1,d_2)\leq  C_{n,d}(d_1-1,d_2+1).
\end{equation}

In order to verify the inequality $(\ref{1.8}),$ we observe that whenever $n\geq 3$ one has
$$(n-1)\cdot\frac{(n+d_1-2)!}{d_1!}\leq (n-1)\cdot \frac{(n+d_2-1)!}{(d_2+1)!}.$$
Equivalently, this is seen to be
$$\frac{(n+d_1-1)!}{d_1!}-\frac{(n+d_1-2)!}{(d_1-1)!}\leq \frac{(n+d_2)!}{(d_2+1)!}-\frac{(n+d_2-1)!}{d_2!},$$
and thus
\begin{equation}\label{1.9}
    \frac{(n+d_1-1)!}{d_1!}+\frac{(n+d_2-1)!}{d_2!}\leq \frac{(n+d_2)!}{(d_2+1)!}+\frac{(n+d_1-2)!}{(d_1-1)!}.
\end{equation}
Therefore, we find from $(\ref{1.9})$ that 
$$\frac{(n+d-1)!}{d!}C_{n,d}(d_1,d_2)\leq \frac{(n+d-1)!}{d!}C_{n,d}(d_1-1,d_2+1). $$ This confirms the inequality $(\ref{1.8}).$ Then, by applying $(\ref{1.8})$ iteratively, we conclude that the quantity $C_{n,d}$ attains the maximum value when $(d_1, d_2) = (1, d-1)$ or $(d_1, d_2) = (d-1, 1)$.

Next, we deduce by applying Lemma $\autoref{lem.factorialcomparison}$ that
\begin{equation}\label{eq.factorials}
\begin{split}
C_{n,d}(1,d-1)& = \frac{d!n!}{(n+d-1)!}+\frac{d!(n+d - 2)!}{(d-1)!(n+d-1)!} \\
& =  \frac{d!n!}{(n+d-1)!}+\frac{d}{n+d-1} \\
&= \frac{n+d}{{n+d \choose d}} +\frac{d}{n+d-1} \\
& <\frac{(n-1)(n+d)}{(n+d)(n+d-1)} + \frac{d}{n+d-1} \\
&= \frac{n-1}{n+d-1} + \frac{d}{n+d-1} = 1.
\end{split}
\end{equation}
Therefore, this completes the proof of Lemma $\ref{lem.Cndmaximum}$.
\end{proof}

%\todo{Seungu: do we fix $\a?$ Is $I(A)$ supposed to depend on $P$? \\
%\ \ \ Kiseok: It would be better to define 
%$I(A)=\{V_{\a}\in \P^{n-1}|\ \langle\a,\nu_{d,n}(\x)\rangle\ \text{is reducible}\}$ or $I(A)=\{\a\in [-A,A]^N|\ \langle\a,\nu_{d,n}(\x)\rangle\ \text{is reducible}\}$ . So, $I(A)$ may not depend on $P$.}

\begin{prop}\label{reducible coefficients locus} 
%Let $n\geq 3$, $d\geq 3$ be integers and $P(t_1,\dots t_N)\in\mathbb{Z}[t_1,\dots, t_N]$ be a homogeneous polynomial of degree $k$. Suppose that $2 \le k<\lfloor (1-C_{n,d}(1,d-1))N\rfloor$. 
Recall the definition of $Y\subseteq \bA^N_\Z$ in (\ref{Eqn.defnofY}). Then, $Y$ is an affine variety. Further, let $r$ be the codimension of $Y$. Then, one has $ r> (1-C_{n,d}(1,d-1))N-1$.

\begin{comment}
For $n\geq 3$, $d\geq 3$, and $\deg P\leq \lfloor (1-C_{n,d}(1,d-1))N\rfloor$, one has

\[\lim_{A\to\infty} \frac{\#I(A)}{\#\mathbb{V}^P_{d,n}(A)}=0\]
\end{comment}

\end{prop}

%\todo{Kiseok: By the circle method, we have $\#\V^{P}_{d,n}(A)\asymp A^{N-k}$. So, it is enough to consider $\#I(A).$ This part may provide the condition on degree $k$ of polynomial $P$ about how large this $k$ could be. Thank you!\\
%Seungsu: Is $I(A)\subseteq \bR^N$?\\
%Kiseok: $I(A)\subseteq \Z^N$
%}

\begin{proof}

%We prove the theorem by arguing by dimension. Indeed, by [\ref{ref8}], we have $\#\V^{P}_{d,n}(A)\sim A^{N-\deg P}$. 

%We aim to show that $\dim Y < N-\deg P$. 
Let $\langle \a , v_{d,n}(\x)\rangle$ be a reducible polynomial. Then, we write $\langle \a , v_{d,n}(\x)\rangle = f^{(1)}(\x)f^{(2)}(\x)$. Here, $f^{(1)}(\x)$ and $f^{(2)}(\x)$ are homogeneous whose degrees are strictly less than $d$. Let $d_i = \deg(f^{(i)}(\x))$ and $t={n+d_1 - 1 \choose n-1}$. Then,
\begin{equation}\label{eq.reducible} \langle \a , v_{d,n}(\x)\rangle = \underbrace{(u_1x_1^{d_1}+ \cdots u_tx_n^{d_1})}_{=f^{(1)}(\x)}\underbrace{(u_{t+1}x_1^{d_2}+\cdots + u_{M(d_1,d_2)} x_n^{d_2})}_{=f^{(2)}(\x)}\end{equation}

Let $Y_{(d_1, d_2)} \subset Y$ be a subset where $\langle \a , v_{d,n}(\x)\rangle$ seperates as (\ref{eq.reducible}). Comparing the coefficients in (\ref{eq.reducible}) for both sides, we attain $a_i$'s as a polynomial of $u_1, \dots, u_M$. Now, let us write $a_i = g_i(u_1, \dots, u_{M(d_1,d_2)})$. Consider a map $\varphi_{(d_1,d_2)}: \Z[t_1, \dots, t_N]\to \Z[u_1, \dots, u_{M(d_1,d_2)}]$ sending $t_i \mapsto g_i(u_1, \dots, u_{M(d_1,d_2)})$. Then, by the construction, $Y_{(d_1,d_2)} = V(\ker \varphi_{(d_1,d_2)})$, and so $Y=\bigcup Y_{(d_1,d_2)}$ is an affine variety. 

Now, we prove $r> (1-C_{n,d}(1,d-1))N-1$. We will instead find the upper bound of the dimension of $Y$. Let $M = \max M(d_1, d_2)$. Since $ \Z[t_1, \dots, t_N]/\ker\varphi_{(d_1,d_2)}$ injects into $\Z[u_1,\dots, u_{M(d_1,d_2)}]$, $\dim Y$ is less than or equal to $M$. Hence, it suffices to show $M\leq N- (1-C_{n,d}(1,d-1))N$. Note that we have $M(d_1,d_2)=C_{d,n}(d_1,d_2)N \leq  N -  (1-C_{n,d}(1,d-1))N$. Indeed, we have
\[\begin{split}
M(d_1, d_2)& = {n+d_1 - 1 \choose n-1} + {n+d_2 - 1 \choose n-1}\\
 &= \frac{(n+d_1 -1)!}{(n-1)!d_1!}+\frac{(n+d_2 -1)!}{(n-1)!d_2!}\\
 & = \frac{(n+d-1)!}{d!(n-1)!}\underbrace{\left(\frac{d!(n+d_1 - 1)!}{d_1!(n+d-1)!}+\frac{d!(n+d_2 - 1)!}{d_2!(n+d-1)!}\right)}_{C_{n,d}(d_1,d_2)}\\
 &= C_{n,d}(d_1,d_2)N \leq  N- (1-C_{n,d}(1,d-1))N %\\
 %&< \frac{(n+d-1)!}{d!(n-1)!}=N
\end{split}\]
The latter inequality is by Lemma $\ref{lem.Cndmaximum}$. Now, since $\dim Y\leq M \leq N- (1-C_{n,d}(1,d-1))N$, we have $ (1-C_{n,d}(1,d-1))N -1 <N-\dim Y  =r $, which is desired. \qedhere

%It suffices to show $\dim Y< M = \max{M(d_1, d_2)}$. 
 
%\frac{d!(n+d_1 - 1)!}{d_1!(n+d-1)!} = \frac{d(d-1)\cdots(d_1+1)}{(n+d-1)(n+d-2)\cdots (n+d_1)}

\end{proof}

\begin{rmk}
Using the identity of Lemma \ref{lem.factorialcomparison} and Lemma \ref{lem.Cndmaximum}, it is not difficult to show that $(1-C_{n,d}(1,d-1))N -1$ is a positive integer. In fact, if $d\geq 4$, the constant is strictly greater than 2.

%\begin{comment}

Here, in order to justify the positivity, we need to show \[1-\frac{1}{N}>C_{n,d}(1,d-1),\] or equivalently, \begin{equation}\label{eq.goaloverN} 1-C_{n,d}(1,d-1)>\frac{1}{N} \end{equation}
First of all, $N={n+d-1 \choose d}$ and by Lemma \ref{lem.factorialcomparison}, we have \begin{equation}\label{eq.1overNboundwithLemma} \frac{1}{N}< \frac{n-2}{(n+d-1)(n+d-2)}\end{equation} 
(Note: In order to make it strictly greater than 2, i.e., $(1-C_{n,d}(1,d-1))N -1>2$, it suffices to show $ 1-C_{n,d}(1,d-1)>\frac{3}{N}$. i.e., we just need to vary the factor of (\ref{eq.1overNboundwithLemma})).
Hence, it suffices to bound the left-hand side of (\ref{eq.goaloverN}) below by the right-hand side of (\ref{eq.1overNboundwithLemma}). Using the representations in (\ref{eq.factorials}), \[\begin{split} 1-C_{n,d}(1,d-1) & = \left(\frac{n-1}{n+d-1} + \frac{d}{n+d-1}\right) - \left(\frac{n+d}{{n+d \choose d}} +\frac{d}{n+d-1}\right)\\ & = \frac{n-1}{n+d-1}\left(1- \frac{\frac{(n+d)(n+d-1)}{n-1}}{{n+d \choose d}}\right) \\ & \stackrel{(\ast)}{>} \frac{n-1}{n+d-1}\left(1- \frac{n(n-2)}{(n-1)(n+d-2)}\right) \\ &=  \frac{(n-1)(d-1) + 1}{(n+d-1)(n+d-2)} > \frac{n-2}{(n+d-1)(n+d-2)} \end{split}\]
The last inequality is due to the fact that $d\geq 3$ (Here, if $d\geq 4$, the last inequality gets us $1-C_{n,d}(1,d-1)>\frac{3}{N}$). The inequality $(*)$ is true since
\[\frac{(n+d)(n+d-1)}{n-1}<\frac{(n+d)!}{n!d!}\cdot \frac{n(n-2)}{(n-1)(n+d-2)}\] by applying (\ref{eq.InductionComputation}) for $n-1$.

%\end{comment}

\end{rmk}

\section{Proof of Theorem 1.1 and Theorem 1.2}\label{sec3}
In this section, we provide the proofs of Theorem $\ref{thm1.2}$ and $\ref{thm1.21.2}$. To ease the notations, we denote by
\begin{align*}
    T_\infty &:= \pi(V(\R)) \cap [-1,1]^N
    \\
    T_p &:= \pi(V(\Q_p)) \cap \Z_p^N \ \forall p \ \text{prime}.
\end{align*}
%Recall the definition ($\ref{Tinfinity}$) and ($\ref{Tp}$) of sets $T_{\infty}$ and $T_p$ for every prime $p$. 
%\memo{HL: this sentence is identical to the next lemma. let's just delete it.} 
We begin this section with a lemma on the measurability of the sets $T_{\infty}$ and $T_p$.
\begin{lem}\label{lem3.1}
    The sets $T_{\infty}\subseteq \R^N$ and $T_p\subseteq \Z_p^N$ are measurable with the Lebesgue measure and the Haar measure $\mu_p,$ respectively.
\end{lem}
\begin{proof}
%\memo{HL: shortened the proof} 
This follows from a version of Tarski–Seidenberg–Macintyre theorem which implies that $T_{\infty}$ and $T_{p}$ are semialgebraic sets (see [$\ref{ref100.12}$, Theorem 3]). %Hence, we see that these are measurable with respect to the Lebesgue measure and the Haar measure $\mu_p$, respectively.
\end{proof}

In advance of the proofs of Theorem $\ref{thm1.2}$ and Theorem $\ref{thm1.21.2}$, we provide some definitions and observations. We consider the natural map
\begin{align*}
    \Phi^A:\ [-A,A]^N\cap\Z^N&\rightarrow [-1,1]^N\times\displaystyle\prod_{p\ \text{prime}}\Z_p^N \\
    \a &\mapsto\left(\frac{\a}{A},\a,\ldots,\a,\ldots\right).
\end{align*}
We sometimes use a different order of primes in the product $\prod_{p\ \text{prime}}\Z_p$, for notational convenience.

Furthermore, for a given subset $\mathcal{U}$ of $[-1,1]^N\times\displaystyle\prod_{p\ \text{prime}}\Z_p^N$ and a given polynomial $P(\mathbf{t})\in \Z[t_1,\dots, t_N]$, we define a quantity $\d( \mathcal{U},A;P)$ by
\begin{equation}\label{2.22.2}
   \d(\mathcal{U},A;P):=\frac{\#\left\{\a\in (\Phi^A)^{-1}(\mathcal{U})\middle|\ P(\a)=0\right\}}{\#\left\{\a\in [-A,A]^N\cap \Z^N\middle|\ P(\a)=0\right\}}. 
\end{equation}

 We say that a subset of $\Z_p$ is an open interval if it has the form $\{x\in \Z_p|\ |x-a|_p\leq b\}$ for some $a\in \Z_p$ and $b\in \R.$ Furthermore, by an open box $I_p$ in $\Z_p^N$ with a given prime $p,$ we mean a Cartesian product of open intervals. Suppose that $\mathfrak{B}$ be an open box in $[-1,1]^N\cap \R^N.$ For a given set $\mathfrak{p}$ of prime numbers, it follows that
\begin{equation*}
\begin{aligned}
    & \d\biggl(\mathfrak{B}\times\displaystyle\prod_{p\in \mathfrak{p}}I_p\times \displaystyle\prod_{p\notin \mathfrak{p}}\Z_p^N,A;P\biggr)\\
     &=\frac{\#\biggl\{\a\in (\Phi^A)^{-1}\biggl(\mathfrak{B}\times\displaystyle\prod_{p\in \mathfrak{p}}I_p\times \displaystyle\prod_{p\notin \mathfrak{p}}\Z_p^N\biggr)\biggl|\ P(\a)=0\biggr\}}{\#\left\{\a\in [-A,A]^N\cap \Z^N\middle|\ P(\a)=0\right\}}.
\end{aligned}
\end{equation*}
%\memo{HL: here $\Z^N_{\mathrm{prim}}$ should be $\Z^N$??} 
We observe that the set $(\Phi^A)^{-1}\biggl(\mathfrak{B}\times\displaystyle\prod_{p\in \mathfrak{p}}I_p\times \displaystyle\prod_{p\notin \mathfrak{p}}\Z_p^N\biggr)$ can be viewed by a set of integers in $A\mathfrak{B}\cap \Z^N$ satisfying certain congruence conditions associated with the radius and the center of the open intervals defining $I_p$. Thus, we infer from the Chinese remainder theorem that there exists $B\in \Z$ whose prime divisors are in $\mathfrak{p}$, and $\r\in\Z^N$ such that
\begin{equation*}
    \d\biggl(\mathfrak{B}\times\displaystyle\prod_{p\in \mathfrak{p}}I_p\times \displaystyle\prod_{p\notin \mathfrak{p}}\Z_p^N,A;P\biggr)=\frac{\#\biggl\{\a\in A\mathfrak{B}\cap \Z^N\biggl|\ P(\a)=0,\ \a\equiv\r\ \text{mod}\ B\biggr\}}{\#\left\{\a\in [-A,A]^N\cap \Z^N\middle|\ P(\a)=0\right\}}.
\end{equation*}
Then, by applying Lemma $\ref{lem1.31.31.3}$, we obtain
 \begin{equation}\label{3.7}
 \begin{aligned}
    &\d\biggl(\mathfrak{B}\times\displaystyle\prod_{p\in \mathfrak{p}}I_p\times \displaystyle\prod_{p\notin \mathfrak{p}}\Z_p^N,A;P\biggr)\\
    &=\frac{\frac{1}{\zeta(N-k)}\prod_{p\notin \mathfrak{p}}\sigma_p\cdot \prod_{p\in \mathfrak{p}}\sigma_p^{B,\r}\cdot \sigma_{\infty}(\mathfrak{B})+O(A^{-\delta})}{\frac{1}{\zeta(N-k)}\prod_{p}\sigma_p \cdot  \sigma_{\infty}([-1,1]^N)+O(A^{-\delta})}.
 \end{aligned}
\end{equation}
For a given measurable set $S_p\subseteq \Z_p^N,$ recall the definition of $\sigma_p(S_p)$ in the preamble to the statement of Theorem $\ref{thm1.2}.$ We observe that $\sigma_p^{B,\r}=\sigma_p(I_p)$ and $\sigma_p=\sigma_p(\Z_p)$. Therefore, we find from ($\ref{3.7}$) that
\begin{equation}\label{key}
    \lim_{A\rightarrow \infty}\d\biggl(\mathfrak{B}\times\displaystyle\prod_{p\in \mathfrak{p}}I_p\times \displaystyle\prod_{p\notin \mathfrak{p}}\Z_p^N,A;P\biggr)=\frac{\prod_{p\in \mathfrak{p}}\sigma_{p}(I_p)\cdot\sigma_{\infty}(\mathfrak{B})}{\prod_{p\in \mathfrak{p}}\sigma_{p}(\Z_p)\cdot\sigma_{\infty}([-1,1]^N)}
\end{equation}
 
%Also, we recall the definition of $B_p(\b,\eta)$ and $B_{\infty}(\b,\eta)$ in the preamble to Proposition $\ref{lem1.3}.$ 

% We note that for $\b\in \Z^N$ and $\eta<1,$ there exists $C:=C(\b)>0$ such that $B_{\infty}(\b/C,\eta/C)\subseteq [-1,1]^N.$ 

\begin{proof}[Proof of Theorem 1.1]
    Recall that
    \begin{align*}
   T_{\infty}&=\left\{\a\in [-1,1]^N\cap \R^N\middle|\ \exists\ \x\in \R^n\setminus\{\boldsymbol{0}\}\ \text{such that}\ f_{\a}(\x)=0\right\} \\
   T_p&=\left\{\a\in \mathbb{Z}_p^N\middle|\  \exists\ \x\in \Z_p^n\setminus\{\boldsymbol{0}\}\ \text{such that}\ f_{\a}(\x)=0\right\}.
\end{align*}
On recalling the definition ($\ref{2.22.2}$) of $\boldsymbol{d}(\cdot, A;P),$
we infer that 
$$\varrho_{d,n}^{P,\text{loc}}(A)=\boldsymbol{d}\biggl(T_{\infty}\times \displaystyle\prod_{p\ \text{prime}}T_p,A;P\biggr).$$ Thus, it suffices to show that
\begin{equation}\label{goal}
   \lim_{A\rightarrow\infty} \boldsymbol{d}\biggl(T_{\infty}\times \displaystyle\prod_{p\ \text{prime}}T_p,A;P\biggr)=c_P.
\end{equation}

In order to verify that the equality $(\ref{goal})$ holds, we introduce
$$\boldsymbol{d}(A,M)=\boldsymbol{d}\biggl(T_{\infty}\times \displaystyle\prod_{p<M}T_p\times \displaystyle\prod_{p\geq M}\Z_p ,A;P\biggr).$$ %Meanwhile, for any given open boxes $\mathfrak{B}\subseteq [-1,1]^N\cap \R^N$ and $I_p\subseteq \Z_p^N$ with a side length $\eta_p>0$, we infer by applying Lemma $\ref{lem1.31.31.3}$ together with the treatment leading from ($\ref{2.117}$) to ($\ref{zeta}$) that 
%\begin{equation*}
%\begin{aligned}
 %   &\boldsymbol{d}\biggl(\mathfrak{B}\times \displaystyle\prod_{p<M}I_p\times\displaystyle\prod_{p\geq M}\Z_p ,A;P\biggr)\\
 %   &=\frac{\frac{1}{\zeta(N-k)}\cdot\prod_{p<M}\sigma_{p}(I_p)\cdot\prod_{p\geq M}\sigma_{p}(\Z_p)\cdot\sigma_{\infty}(\mathfrak{B})+O(A^{-\delta})}{\frac{1}{\zeta(N-k)}\cdot\prod_{p}\sigma_{p}(\Z_p)\cdot\sigma_{\infty}([-1,1]^N)+O(A^{-\delta})}.
%\end{aligned}
%\end{equation*}
%Hence, by letting $A\rightarrow \infty$, we find that 
%\begin{equation}\label{limdensity}
 %   \lim_{A\rightarrow \infty}\boldsymbol{d}\biggl(\mathfrak{B}\times \displaystyle\prod_{p<M}I_p\times\displaystyle\prod_{p\geq M}\Z_p,A;P\biggr)=\frac{\prod_{p<M}\sigma_{p}(I_p)\cdot\sigma_{\infty}(\mathfrak{B})}{\prod_{p<M}\sigma_{p}(\Z_p)\cdot\sigma_{\infty}([-1,1]^N)}.
%\end{equation}
%Then, on noting from Lemma $\ref{lem3.1}$ that $T_{\infty}$ and $T_p$ are measurable with the Lebesgue measure and the Haar measure $\mu_p$, one infers from $(\ref{limdensity})$ that
%\begin{equation}\label{3.73.7}
%\begin{aligned}
 %   \lim_{A\rightarrow \infty}\boldsymbol{d}(A,M)&= \lim_{A\rightarrow \infty}\boldsymbol{d}\biggl(T_{\infty}\times \displaystyle\prod_{p<M}T_p\times\displaystyle\prod_{p\geq M}\Z_p,A;P\biggr)\\
 %   &=\frac{\prod_{p<M}\sigma_{p}(T_p)\cdot\sigma_{\infty}(T_{\infty})}{\prod_{p<M}\sigma_{p}(\Z_p)\cdot\sigma_{\infty}([-1,1]^N)}.
%\end{aligned}
%\end{equation}
One sees by applying the triangle inequality that
\begin{equation}\label{triangle}
\begin{aligned}
  & \lim_{A\rightarrow \infty}\biggl|\boldsymbol{d}\biggl(T_{\infty}\times \displaystyle\prod_{p\ \text{prime}}T_p,A;P\biggr)-c_P\biggr| \leq\lim_{A\rightarrow \infty}|d_1(A,M)|+ \lim_{A\rightarrow \infty}|d_2(A,M)|,
\end{aligned}
\end{equation}
where
$$d_1(A,M)=\boldsymbol{d}\biggl(T_{\infty}\times \displaystyle\prod_{p\ \text{prime}}T_p,A;P\biggr)-\boldsymbol{d}(A,M)$$
and
$$d_2(A,M)=\boldsymbol{d}(A,M)-c_P.$$

First, we analyze the quantity $$\lim_{A\rightarrow \infty}|d_1(A,M)|.$$
We readily see from the definition of $\boldsymbol{d}(A,M)$ that
$$|d_1(A,M)|=\boldsymbol{d}\biggl(T_{\infty}\times \prod_{p<M}T_p\times \biggl(\prod_{p\geq M}T_p\biggr)^c, A;P\biggr).$$
Furthermore, we find that
\begin{equation}\label{d1AM}
\begin{aligned}
   &|d_1(A,M)| \\
   &\leq \boldsymbol{d}\biggl([-1,1]^N\times \prod_{p<M}\Z_p\times \biggl(\prod_{p\geq M}T_p\biggr)^c, A;P\biggr)\\
   &\leq \frac{\left\{\a\in [-A,A]^N\cap \Z^N\middle|\ \begin{aligned}
        &(i)\ \exists\ p>M\ \text{s.t.}\ f_{\a}(\x)=0\ \text{has no solution in}\ \Z_p^n\\
        &(ii)\ P(\a)=0
    \end{aligned}\right\}}{\#\left\{\a\in [-A,A]^N\cap \Z^N\middle|\ P(\a)=0\right\}}.
\end{aligned}
\end{equation}
Meanwhile, for sufficiently large prime $p$, whenever $f_{\a}(\x)$ is irreducible over $\overline{\mathbb{F}}_p$, the Lang-Weil estimate [$\ref{ref100.2}$] (see also $[\ref{ref100.0}, \text{Theorem}\ 3]$) ensures the existence of a smooth point $\x\in ({\mathbb{F}}_p)^n$ satisfying $f_{\a}(\x)=0$. Then, by Hensel's lemma, we have a point $\x$ in $\Q_p$ satisfying $f_{\a}(\x)=0.$ Therefore, we conclude from $(\ref{d1AM})$ that for sufficiently large $M>0$, one has 
\begin{equation}\label{3.113.11}
\begin{aligned}
    &|d_1(A,M)|\\
    &\leq \frac{\left\{\a\in [-A,A]^N\cap \Z^N\middle|\ \begin{aligned}
        &(i)\ \exists\ p>M\ \text{s.t.}\ f_{\a}(\x)\ \text{is reducible over}\ \overline{\mathbb{F}}_p\\
        &(ii)\ P(\a)=0
    \end{aligned}\right\}}{\#\left\{\a\in [-A,A]^N\cap \Z^N\middle|\ P(\a)=0\right\}}.
\end{aligned}
\end{equation}
We shall apply Lemma $\ref{lem2.42.4}$ with $Y$ defined in ($\ref{Eqn.defnofY}$). With this $Y$ in mind, we find from $(\ref{3.113.11})$ that
\begin{equation}
    \begin{aligned}
       &|d_1(A,M)|\\
       &\leq   \frac{ \#\left\{\a\in [-A,A]^N\cap \Z^N\middle|\ \begin{aligned}
        &(i) \ \a\ (\text{mod}\ p)\in Y(\mathbb{F}_p)\ \text{for some prime}\ p>M \\
        &(ii)\ P(\a)=0
     \end{aligned}\right\}}{\#\left\{\a\in [-A,A]^N\cap \Z^N\middle|\ P(\a)=0\right\}}.
    \end{aligned}
\end{equation}
Note by the classical argument (see $[\ref{ref8}]$ and $[\ref{ref26}, \text{the proof of Theorem 1.3}]$) that the fact that $P(\t)=0$ has a nontrivial integer solution implies that $ \sigma_{\infty}([-1,1]^N)\cdot\prod_{p} \sigma_p(\Z_p)\asymp 1$. Then, one infers by Lemma $\ref{lem1.31.31.3}$ that $$\#\left\{\a\in [-A,A]^N\cap \Z^N\middle|\ P(\a)=0\right\}\asymp A^{N-k}.$$ Proposition $\ref{reducible coefficients locus}$ together with the hypothesis in the statement of Theorem $\ref{thm1.2}$ that $k<\lfloor (1-C_{n,d}(1,d-1))N\rfloor$ reveals that the codimension $r$ of $Y$ is strictly greater than $k+1$. Hence, we find by applying Lemma $\ref{lem2.42.4}$ that 
\begin{equation*}
    |d_1(A,M)|\ll  \frac{1}{M^{r-k-1}\log M}+A^{k-r+1}.
\end{equation*}
Therefore, we obtain
\begin{equation}\label{3.123.12}
    \lim_{A\rightarrow \infty}|d_1(A,M)|\ll \frac{1}{M^{r-k-1}\log M}.
\end{equation}

Next, we turn to estimate $\lim_{A\rightarrow \infty}|d_2(A,M)|.$ For simplicity, we temporarily write
\begin{equation}\label{3.83.8}
    \lim_{A\rightarrow \infty}|d_2(A,M)|=\varphi(M).
\end{equation}
One infers by $(\ref{key})$ together with Lemma $\ref{lem3.1}$ that 
$$d_2(A,M)=\frac{\prod_{p\leq M}\sigma_{p}(T_p)\cdot\sigma_{\infty}(T_{\infty})}{\prod_{p\in \mathfrak{p}}\sigma_{p}(\Z_p)\cdot\sigma_{\infty}([-1,1]^N)},$$
and thus by the definition of $c_P$, we discern that 
\begin{equation}\label{3.103.10}
    \varphi(M)\rightarrow 0,
\end{equation} as $M\rightarrow \infty.$

Hence, we conclude from $(\ref{triangle}),$ $(\ref{3.123.12})$ and $(\ref{3.83.8})$ that 
$$\lim_{A\rightarrow \infty}\biggl|\boldsymbol{d}\biggl(T_{\infty}\times \displaystyle\prod_{p\ \text{prime}}T_p,A;P\biggr)-c_P\biggr|\ll \frac{1}{M^{r-k-1}\log M}+\varphi(M).$$ By letting $M\rightarrow \infty$, we see from $(\ref{3.103.10})$ that 
$$\lim_{A\rightarrow \infty}\biggl|\boldsymbol{d}\biggl(T_{\infty}\times \displaystyle\prod_{p\ \text{prime}}T_p,A;P\biggr)-c_P\biggr|=0,$$
which gives $(\ref{goal})$. This completes the proof of Theorem $\ref{thm1.2}.$
%Furthermore, by applying the same manner leading to $(\ref{3.73.7})$, it follows that for any $M_1>M$, one has
%\begin{equation*}
%\begin{aligned}
 %    \lim_{A\rightarrow \infty}\boldsymbol{d}\left([-1,1]^N\times \displaystyle\prod_{M\leq p<M_1}T_p\times \displaystyle\prod_{p\notin [M,M_1)}\Z_p,A;P\right)=\frac{\prod_{M\leq p<M_1}\sigma_{p}(T_p)}{\prod_{M\leq p<M_1}\sigma_{p}(\Z_p)}.
%\end{aligned}
%\end{equation*}
\end{proof}

\begin{rmk}
  By making use of Lemma $\ref{lem2.42.4}$, it seems possible to apply [$\ref{ref1511}$, Theorem 2.5] in order to obtain the conclusion of Theorem 1.1. However, in order to make this paper self-contained and for readers who are interested in an explicit way for particular thin sets we are dealing with, we record the procedure in full.
\end{rmk}

%\bigskip

For the proof of Theorem $\ref{thm1.21.2}$, we require a proposition that plays an important role in guaranteeing the positiveness of $c_P$.
In order to describe this proposition, it is convenient to define $S_{\infty}:=\pi(V(\R)).$
Also, for given points $\b_p\in \Z_p^N$ and $\b_{\infty}\in \R^N$, we define
\begin{equation*}
    \begin{aligned}
        B_{\infty}(\b_{\infty},\eta)&=\left\{\a\in \R^N\middle|\ |a_i-(\b_{\infty})_i|<\eta\ \text{for $1\leq i\leq N$}\right\}\\
        B_{p}(\b_p,\eta)&=\left\{\a\in \Z_p^N\middle|\ |a_i-(\b_p)_i|_{p}<\eta\ \text{for $1\leq i\leq N$}\right\},
    \end{aligned}
\end{equation*}
for every prime $p.$ %Furthermore, we consider two functions $\Psi_{\infty}:\R^N\times \R^n\rightarrow \R$ and $\Psi_p:\Z_p^N\times \Z_p^n\rightarrow \Z_p$  defined by 
%\begin{equation}\label{funcs1.1}
 % \Psi_{\infty}(\a,\x):=f_{\a}(\x)\ \text{and}\ \Psi_p(\a,\x):=f_{\a}(\x),  
%\end{equation}
%for every prime $p.$
%For the gradient vectors associated with the polynomial $f_{\a}(\x)$, we use the notation
%\begin{equation}\label{notat1.2}
%    \begin{aligned}
 %       &\nabla f_{\a}(\x)=(\partial_{x_1}f_{\a}(\x),\ldots,\partial_{x_n}f_{\a}(\x)).
 %   \end{aligned}
%\end{equation}

 \begin{prop}\label{lem1.3}
Suppose that there exists $\b_p\in \Z_p^N\setminus\{\0\}$ (resp.~$\b_\infty \in \R^N\setminus \{0\}$) such that the variety $V_{\b_p}$ in $\P^{n-1}_{\Q_p}$ (resp.~$V_{\b_{\infty}}$ in $\P^{n-1}_\R$) admits a smooth $\Q_p$-point (resp.~$\R$-point) for each prime $p$. Then, for each prime $p$, there exists positive numbers $\eta_p$  and  $\eta_{\infty}$ less than or equal to $1$ such that 
\begin{align*}
     B_{p}(\b_p,\eta_p)&\subseteq T_{p} \\
     B_{\infty}(\b_\infty,\eta_{\infty})&\subseteq S_{\infty}.
\end{align*}
 \end{prop}

 \begin{proof}
     By the existence of a smooth $\Q_p$-point in $V_{\b_p}$, there is an open neighborhood of $V$ containing the smooth $\Q_p$-point such that the restriction of $\pi$ (after base change to $\Q_p$) to the neighborhood is a smooth map into a neighborhood of $\A^N$ containing $\b_p$. By \cite[Theorem 10.5.1]{CT-S}, this map induces a topologically open map between $\Q_p$-points, and thus the image contains an open ball $B_{p}(\b_p,\eta_p)$ for some $\eta_p$. The proof for the real place is identical.
 \end{proof}
%\bigskip

%\memo{HL: also here, $\Z^N_{\mathrm{prim}}$ should be $\Z^N$?}
\begin{proof}[Proof of Theorem 1.2]
    Recall the natural map
    \begin{align*}
        \Phi^A:\ [-A,A]^N\cap\Z^N&\rightarrow [-1,1]^N\times\displaystyle\prod_{p\ \text{prime}}\Z_p^N \\
        \a&\mapsto \left(\frac{\a}{A},\a,\ldots,\a,\ldots\right).
    \end{align*}
%\memo{HL: I think it is better not writing all these definitions repeatedly.} 
Furthermore, we recall the definitions of $S_{\infty}:=\pi(V(\R))$ and $ T_p := \pi(V(\Q_p)) \cap \Z_p^N$ for all primes $p.$

Recall from the hypothesis in the statement of Theorem $\ref{thm1.21.2}$ that for each place $v$ of $\Q$, there exists a non-zero $\boldsymbol{b}_v\in \Q_v^N$ such that $P(\boldsymbol{b}_v)=0$ and the variety $V_{\boldsymbol{b}_v}$ in $\mathbb{P}^{n-1}$ admits a smooth $\mathbb{Q}_v$-point. For $v=p$ $(p\ \textrm{prime}),$ we assume that this $\boldsymbol{b}_p$ is a $\Z_p$-point. Otherwise, by multiplying powers of $p$, we can make it a $\Z_p$-point not changing the variety $V_{\b_p}$. One sees by applying Proposition  \ref{lem1.3} that for each prime $p$ there exist positive numbers $\eta_{p}$ and $\eta_{\infty}$ less than $1$ such that $B_p(\b_p,\eta_p)\subseteq T_p$ and $B_{\infty}(\b_{\infty},\eta_{\infty})\subseteq S_{\infty}.$  We choose a sufficiently large number $C:=C(\b_{\infty})>0$ such that $B_{\infty}(\b_{\infty}/C,\eta_{\infty}/C)\subseteq [-1,1]^N.$ Furthermore, on observing the relation that $f_{\b_{\infty}/C}(\x)=(1/C)\cdot f_{\b_{\infty}}(\x)$, we infer that $B_{\infty}(\b_{\infty}/C,\eta_{\infty}/C)\subseteq S_{\infty}.$ Therefore, on noting that
\begin{equation*}
    \begin{aligned}
        T_{\infty}=S_{\infty}\cap [-1,1]^N.
    \end{aligned}
\end{equation*}  one deduces that $$B_{\infty}(\b_{\infty}/C,\eta/C)\subseteq T_{\infty}.$$

Meanwhile, it follows by Theorem $\ref{thm1.2}$ that 
\begin{equation*}
    \lim_{A\rightarrow \infty}\varrho_{d,n}^{P,\text{loc}}(A)=c_P,
\end{equation*}
and thus, we find that
\begin{equation}\label{lowerbound}
     \lim_{A\rightarrow \infty}\varrho_{d,n}^{P,\text{loc}}(A)\geq \frac{\prod_{p<M} \sigma_{p}(B_p(\b_p,\eta_p))\cdot \prod_{p\geq M}\sigma_{p}(T_p)\cdot \sigma_{\infty}(B_{\infty}(\b_{\infty}/C,\eta_{\infty}/C))}{\prod_p \sigma_p(\Z_p)\cdot \sigma_{\infty}([-1,1]^N)},
\end{equation}
for any $M>0.$
Then, it suffices to show that the right-hand side in $(\ref{lowerbound})$ is greater than $0$. We shall prove this by showing that there exists $M>0$ such that
\begin{equation}\label{11}
    \frac{\prod_{p<M} \sigma_{p}(B_p(\b_p,\eta_p))\cdot\sigma_{\infty}(B_{\infty}(\b_{\infty}/C,\eta_{\infty}/C))}{\prod_{p<M} \sigma_p(\Z_p)\cdot \sigma_{\infty}([-1,1]^N)}>0
\end{equation}
and 
\begin{equation}\label{22}
    \frac{\prod_{p\geq M}\sigma_p(T_p)}{\prod_{p\geq M}\sigma_p(\Z_p)}>0.
\end{equation}

First, we shall show that the inequality $(\ref{11})$ holds. For a given $B\in\N$ and $\r\in \Z^N$, we recall the definition of $\sigma_p^{B,\r}$ in the statement of Lemma $\ref{lem1.31.31.3}$. Note that there exists $B\in\N\cup\{0\}$ such that 
\begin{equation}\label{equal}
\displaystyle\prod_{p<M}\sigma_p(B_p(\b_p,\eta_p))=\displaystyle\prod_{p<M}\sigma_p^{B,\b_p}.
\end{equation} Furthermore, the $p$-adic densities $\sigma_p(\Z_p)$, $\sigma_p^{B,\b_p}$ and the real densities $\sigma_{\infty}([-1,1]^N)$, $\sigma_{\infty}(B_{\infty}(\b_{\infty}/C,\eta_{\infty}/C))$ are greater than $0$ by the application of the Hensel's lemma and the implicit function theorem (see [$\ref{ref26}$, the proof of Theorem 1.3] or [$\ref{ref3}$, Lemma 5.7]). Therefore, one sees from ($\ref{equal}$) that the inequality $(\ref{11})$ holds for any $M>0$.

Next, we shall show that $(\ref{22})$ holds. For any $M_1$ with $M_1>M,$ we find from $(\ref{key})$ with $[M,M_1)$ and $[-1,1]^N$ in place of $\mathfrak{p}$ and $\mathfrak{B}$ that 
\begin{equation}\label{final1}
\begin{aligned}
 \lim_{A\rightarrow \infty}\d\biggl([-1,1]^N\times\displaystyle\prod_{p\in [M,M_1)}T_p\times \displaystyle\prod_{p\notin [M,M_1)}\Z_p^N,A;P\biggr)=\frac{\prod_{p\in [M,M_1)}\sigma_{p}(T_p)}{\prod_{p\in [M,M_1)}\sigma_{p}(\Z_p)}.
\end{aligned}
\end{equation}
Meanwhile, we see that
\begin{equation*}
\begin{aligned}
   &\d\biggl([-1,1]^N\times\biggl(\displaystyle\prod_{p\in [M,M_1)}T_p\biggr)^c\times \displaystyle\prod_{p\notin [M,M_1)}\Z_p^N,A;P\biggr) \\
   &\leq \frac{\left\{\a\in [-A,A]^N\cap \Z^N\middle|\ \begin{aligned}
        &(i)\ \exists\ p>M\ \text{s.t.}\ f_{\a}(\x)=0\ \text{has no solution in}\ \Z_p^n\\
        &(ii)\ P(\a)=0
    \end{aligned}\right\}}{\#\left\{\a\in [-A,A]^N\cap \Z^N\middle|\ P(\a)=0\right\}}.
\end{aligned}
\end{equation*}
Then, it  follows by the same argument leading from ($\ref{d1AM}$) to $(\ref{3.123.12})$ that 
\begin{equation}\label{3.2929}
    \begin{aligned}
\lim_{A\rightarrow\infty}\d\biggl([-1,1]^N\times\biggl(\displaystyle\prod_{p\in [M,M_1)}T_p\biggr)^c\times \displaystyle\prod_{p\notin [M,M_1)}\Z_p^N,A;P\biggr)\ll \frac{1}{M^{r-k-1}\log M}.
    \end{aligned}
\end{equation}
Thus, on noting that
\begin{equation*}
\begin{aligned}
    &\d\biggl([-1,1]^N\times\displaystyle\prod_{p\in [M,M_1)}T_p\times \displaystyle\prod_{p\notin [M,M_1)}\Z_p^N,A;P\biggr)\\
    &=1-\d\biggl([-1,1]^N\times\biggl(\displaystyle\prod_{p\in [M,M_1)}T_p\biggr)^c\times \displaystyle\prod_{p\notin [M,M_1)}\Z_p^N,A;P\biggr),
\end{aligned}
\end{equation*}
we find from ($\ref{3.2929}$) that
\begin{equation}\label{final2}
    \lim_{A\rightarrow \infty}\d\biggl([-1,1]^N\times\displaystyle\prod_{p\in [M,M_1)}T_p\times \displaystyle\prod_{p\notin [M,M_1)}\Z_p^N,A;P\biggr)>1/2,
\end{equation}
for sufficiently large $M>0.$ Therefore, it follows from $(\ref{final1})$ and $(\ref{final2})$ that for sufficiently large $M>0,$ one has
\begin{equation*}
\begin{aligned}
    &\frac{\prod_{p\geq M}\sigma_p(T_p)}{\prod_{p\geq M}\sigma_p(\Z_p)}\\&=\lim_{M_1\rightarrow \infty} \frac{\prod_{p\in [M,M_1)}\sigma_{p}(T_p)}{\prod_{p\in [M,M_1)}\sigma_{p}(\Z_p)}\\
    &=\lim_{M_1\rightarrow\infty}\lim_{A\rightarrow \infty}\d\biggl([-1,1]^N\times\displaystyle\prod_{p\in [M,M_1)}T_p\times \displaystyle\prod_{p\notin [M,M_1)}\Z_p^N,A;P\biggr)>1/2.
\end{aligned}
\end{equation*}
Hence, the inequality $(\ref{22})$ holds. By using the inequalities $(\ref{11})$ and $(\ref{22})$, one finds that
$$\frac{\prod_{p<M} \sigma_{p}(B_p(\b_p,\eta_p))\cdot \prod_{p\geq M}\sigma_{p}(T_p)\cdot \sigma_{\infty}(B_{\infty}(\b_{\infty}/C,\eta_{\infty}/C))}{\prod_p \sigma_p(\Z_p)\cdot \sigma_{\infty}([-1,1]^N)}>0,$$
and thus we conclude from ($\ref{lowerbound}$) that 
\begin{align*}
     \lim_{A\rightarrow \infty}\varrho_{d,n}^{P,\text{loc}}(A)>0.
\end{align*}
\end{proof}

\end{document}